\documentclass[10pt]{amsart}
\usepackage{commath}
\usepackage{stmaryrd}
\usepackage{xcolor}
\usepackage[hidelinks]{hyperref}
\newcommand{\N}{\mathbb{N}}

\newcommand{\C}{\mathbb{C}}
\newcommand{\R}{\mathbb{R}}

\newcommand{\Ha}{\mathbb{H}}
\newcommand{\Ho}{\mathbb{H}^3}
\newcommand{\BB}{\mathbb{B}}

\newcommand{\F}{\mathcal{F}}
\newcommand{\tr}{\mathrm{tr}}

\newcommand{\End}{\mathrm{End}}

\newcommand{\re}{\mathrm{Re}}
\newcommand{\T}{\mathcal{T}}
\newcommand{\M}{\mathcal{M}}
\renewcommand{\S}{\mathcal{S}}
\renewcommand{\sf}{\mathrm{sf}}
\newcommand{\sk}{\mathrm{sk}}
\newcommand{\SL}{\mathrm{SL}}
\renewcommand{\sl}{\mathfrak{sl}}
\newcommand{\B}{\mathcal{B}}
\renewcommand{\O}{\mathcal{O}}
\newcommand{\id}{\mathrm{id}}

\newcommand{\Teich}{\mathcal{T}}
\newcommand{\SLC}{\mathrm{SL}_2(\mathbb{C})}

\newcommand{\slc}{\mathfrak{sl}_2(\mathbb{C})}
\newcommand{\slo}{\mathfrak{sl}_2}
\newcommand{\Syst}{\mathcal{S}}
\newcommand{\Ne}{\mathcal{N}}
\newcommand{\Rep}{\mathcal{R}}
\newcommand{\RH}{\mathrm{RH}}
\newcommand{\dett}{\mathbf{det}}

\newcommand{\wt}{\widetilde}
\newcommand{\wh}{\widehat}

\newcommand{\QD}{\operatorname{QD}}
\newcommand{\QDD}{\mathcal{QD}}
\newcommand{\Belt}{\operatorname{Belt}}
\newcommand{\NAH}{\mathrm{NAH}}
\renewcommand{\S}{\mathcal{S}}
\renewcommand{\sf}{\mathrm{sf}}
\renewcommand{\sl}{\mathfrak{sl}}
\renewcommand{\O}{\mathcal{O}}
\newcommand{\can}{\mathrm{can}}
\newcommand{\Ad}{\mathrm{Ad}}
\newcommand{\ur}{\mathrm{reg}}
\newcommand{\Mod}{\mathrm{Mod}}
\newcommand{\df}{\mathrm{df}}
\newcommand{\Orb}{\mathrm{Orb}}
\newcommand{\Real}{\mathrm{Re}}

\newtheorem{theorem}{Theorem}[section]

\newtheorem{lemma}{Lemma}[section]
\newtheorem{claim}{Claim}[section]
\newtheorem{definition}{Definition}[section]

\newtheorem{proposition}{Proposition}[section]
\newtheorem{remark}{Remark}[section]

\begin{document}
\title{Holomorphic curves in compact quotients $\SLC/\Lambda$}
\author{Yiran Lin and Vladimir Markovi\'c}
\address{ \newline  YMSC  \newline Tsinghua University  \newline China }

\today


\begin{abstract} We prove that every 
discrete faithful representation $\rho:\pi_1(\Sigma)\to \SLC$ is the monodromy of a holomorphic connection on the trivial rank-2 vector bundle over a Riemann surface. As an application, we answer the question posed by Ghys \cite{ghys} and  Huckleberry-Winkelmann \cite{h-w}  (known as the Margulis' problem) by proving that every compact quotient 
$\SL_2(\C)/\Lambda$ contains a holomorphic curve of genus $g\geq 2$. The main tools we use are the 
Non-Abelian Hodge correspondence, the WKB analysis, and the Morgan-Shalen compactification.
\end{abstract}

\maketitle

\section{Introduction} Let $\Sigma$ be a closed topological surface of genus at least two. By $\pi_1(\Sigma)$ we denote the corresponding surface group. Let 
$$
\Rep=\{ \rho\in \text{Hom}\big(\pi_1(\Sigma), \SLC \big)  \}  \sslash \SLC
$$
be  the character variety consisting of  representations of  $\pi_1(\Sigma)$ into $\SLC$ (if there is no confusion, we do not distinguish between a representation $\rho$ and its equivalence class in $\Rep$).
The subset  $\Rep^s\subset \Rep$ is defined as
$$
\Rep^s=\{ \rho\in \Rep: \,\, \text{$\rho$ is irreducible}\}.
$$
Let $\Teich=\Teich(\Sigma)$ be the Teichm\"uller space.  We let $\Omega^1(X)$ denote the space of Abelian differentials on $X\in \Teich$, and 
$\slo(\Omega^1(X))$  the space of traceless $2\times 2$ matrices of holomorphic 1-forms on $X$.  
The space of $\slo$-systems over $\Sigma$  is defined as
$$
\Syst=\{(X,A): X\in \Teich, \,\, A\in  \slo(\Omega^1(X)) \} \sslash \SLC
$$
where we identify pairs $(X,A)$ and $(X,B)$ if $A$ and $B$ are  conjugated by and element of $\SLC$.

Given $(X,A)\in \Syst$, we consider the trivial holomorphic bundle $X\times\C^2$ equipped with the holomorphic connection $\nabla=d+A$. Then $\nabla$ is a flat  connection inducing the monodromy representation 
$$
\rho_{(X,A)}:\pi_1(\Sigma)\to \SLC,
$$ 
which defines the monodromy map 
$$
\RH:\Syst\to\Rep
$$
by $\RH(X,A)=\rho_{(X,A)}$.  We define $\Syst^s \subset \Syst$ as
$$
\Syst^s=\{(X,A)\in \Syst:  \,\, \rho_{(X,A)}\,\, \text{is irreducible}  \}.
$$ 
Then we have the restricted map
$$
\RH:\Syst^s\to\Rep^s.
$$

\begin{remark}
Calsamiglia-Deroin-Heu-Loray  \cite{c-d-h-l}  proved that $\RH$ is a local diffeomorphism when $\Sigma$ is of genus two. This was later extended by Biswas-Dumitrescu \cite{b-d}, and Markovi\'c-To\v si\'c \cite{m-t}, where it was shown that $\RH$ is generically a local diffeomorphism for every $\Sigma$.
\end{remark}

\subsection{Discrete faithful representations}
It is a difficult problem to decide whether an irreducible representation is a holonomy of a $\slo$-system. Our principal result is that every discrete faithful representation is. 

\begin{theorem}\label{thm-main} Let $\rho:\pi_1(\Sigma)\rightarrow \SL_2(\C)$ be a discrete faithful representation. Then $\rho\in \RH(\Syst^s)$. Moreover,  the set $\RH^{-1}(\rho)$ projects onto a dense subset of  the moduli space $\M(\Sigma)$ of Riemann surfaces homeomorphic to $\Sigma$.
\end{theorem}
\begin{remark} Let  $P:\Syst^s\to \M(\Sigma)$ be the map sending $(X,A)$ to the (unmarked) Riemann surface $X$. The second statement in Theorem \ref{thm-main} says that $P\big( \RH^{-1}(\rho) \big) $ is a dense subset of $\M(\Sigma)$.
\end{remark}  
\begin{remark} Until recently it was not known whether $\RH(\Syst^s)$ contained any discrete faithful representations. It was shown by Biswas-Dumitrescu-Heller-Heller in \cite{b-d-h-h} that $\RH(\Syst^s)$ contains at least one Fuchsian representation.
\end{remark}

\subsection{The Margulis' problem}  Let $\Lambda<\SLC$ denote a co-compact lattice, and consider the compact quotient $\SLC/\Lambda$. If $\Lambda$ is torsion free this quotient is diffeomorphic to a double covering of the  unitary frame bundle over the hyperbolic 3-manifold $\Ho/\Lambda$. Furthermore, $\SLC/\Lambda$ is a complex 3-fold which is not K\"ahler (and therefore not an algebraic complex 3-fold). It is natural to ask whether it contains complex subvarieties.

Huckleberry and Margulis  \cite{h-m} proved that compact 3-folds  $\SLC/\Lambda$  contain no complex hypersurfaces. On the other hand, they contain plenty  elliptic curves arising as quotients of 1-parameter subgroups in 
$\SLC$.  Ghys \cite{ghys}, and Huckleberry-Winkelmann \cite{h-w}, raised the question whether such quotients $\SLC/\Lambda$  contain higher genus curves (this has been termed  the Margulis' problem \cite{c-d-h-l}). 
\begin{remark} Recently,  Biswas-Dumitrescu-Heller-Heller \cite{b-d-h-h-1}  constructed an example of a  compact quotient $\SLC/\Lambda$  containing a holomorphic curve of genus two (in this example the induced monodromy is not faithful).
\end{remark}
Ghys \cite{ghys} reformulated the Margulis' problem into the  question as to whether there exists $X\in \Teich$, and an irreducible holomorphic connection on  the trivial holomorphic vector bundle $X\times\C^2$,  whose monodromy representation sends $\pi_1(\Sigma)$ into a given co-compact lattice  $\Lambda<\SLC$. Combining this with Theorem \ref{thm-main} and the Surface Subgroup Theorem from \cite{k-m}, we show that every compact quotient $\SLC/\Lambda$ contains many holomorphic curves of large genus. Moreover, the monodromy representation $\pi_1(\Sigma)\to \Lambda$ induced by these holomorphic curves is injective (in fact, it can be chosen to be quasifuchsian).

\begin{theorem}\label{thm-main-1}  Let $\Lambda\subset\SL_2(\C)$ be a co-compact lattice. Then if the genus of  
$\Sigma$ is sufficiently large  there exists $X\in \Teich$ admitting a non-constant holomorphic map $f:X\to \SL_2(\C)/\Lambda$. The collection of such Riemann surfaces $X$ forms a dense subset of the moduli space $\M(\Sigma)$.
\end{theorem}
\begin{remark} The holomorphic maps $f:X\to \SL_2(\C)/\Lambda$ we construct induce a monomorphism $f_*:\pi_1(X)\to \Lambda$. 
\end{remark}

\begin{proof} After replacing $\Lambda$ by one its finite index subgroups if necessary, we may assume that $\Lambda$ has no torsion.  By the works of Kahn-Markovi\'c \cite{k-m} and \cite{k-m-1}, there is $g_0\ge 2$ such that  
if $\Sigma$ is of genus at least $g_0$ then there exists a quasifuchsian representation $\sigma:\pi_1(\Sigma)\to \Lambda$. Next, we construct a holomorphic map $f:X\to \SL_2(\C)/\Lambda$ such that the induced monomorphism between the fundamental groups $f_*:\pi_1(\Sigma)\to \Lambda$ agrees with 
$\sigma$.

Let $(X,A)\in \Syst^s$ be such that $\rho_{(X,A)}=\sigma$.  By Theorem \ref{thm-main} there are infinitely many such pairs $(X,A)$. Moreover, the collection of the corresponding Riemann surfaces $X$ is dense in $\M_g$. Let $\wt{X}$ denote the universal cover of $X$. The following lemma well known lemma was proved in the introduction to \cite{c-d-h-l} (see also Lemma 6 in \cite{m-t}).
\begin{lemma}\label{lemma-mt}
There exists a holomorphic map $F:\tilde{X}\to\SLC$ such that $dF+AF=0$, and
\begin{equation}\label{eq-mt}
F(\gamma x)=F(x)\sigma(\gamma)^{-1},\quad\quad \forall \gamma\in \pi_1(\Sigma).
\end{equation}
\end{lemma}

Now, let $\sigma \big( \pi_1(\Sigma) \big)=\Gamma<\Lambda$ be the corresponding quasifuchsian group. The equivariance condition (\ref{eq-mt}) means that $F$ descends to  the  quotient map  $X\to \SLC/\Gamma$. Composing this map with the  covering $\SLC/\Gamma \to \SLC/\Lambda$ yields the required holomorphic map $f:X\to \SLC/\Lambda$. 
\end{proof}

\subsection{The Higgs bundle parametrisation of $\Rep$}  In this paper we exploit the identification of the character variety $\Rep$ with the moduli space of Higgs bundles. This enables us to study the asymptotic behaviour of the Riemann-Hilbert map $\RH$ using the structures of the moduli space of Higgs bundles.

In the remainder  of this paper, we fix  a non-hyperelliptic $X_0\in \Teich$ (unless the genus of $\Sigma$ is  two  in which case we fix any $X_0\in \Teich$).
We let  $\M$  denote the moduli space of semistable $\SLC$-Higgs bundles $(E,\theta)$ on $X_0$. Here $E$ is a rank-2 holomorphic vector bundle with trivial determinant, and $\theta$  a traceless holomorphic section of the endomorphism bundle of $E$ twisted by the canonical bundle over $X_0$ (the section $\theta$ is called a Higgs field). 

The moduli space $\M$ is an irreducible, quasi projective variety. Its smooth locus is the open subset $\M^s$ parametrising stable Higgs bundles. By the non abelian Hodge theory there is a homeomorphism
\[\NAH:\Rep \to \M.\]
We define the map
\begin{equation}\label{eq-F}
F:\Syst\rightarrow\M
\end{equation}
by letting $F=\NAH\circ \RH$. Our efforts will be focused on studying the asymptotic properties of the map $F$.
The main step towards proving Theorem \ref{thm-main} is Theorem \ref{thm-3} which says that the image of $F$ contains a generic  Higgs bundle $(E,\theta)\in \M^s$ of sufficiently large determinant.

\subsection{The Hitchin fibration} We let $\QD(X)$ denote the vector space of holomorphic quadratic differentials on $X\in \Teich$. In the special case $X=X_0$, we introduce the notation
$$
\B \stackrel{\text{def}}{=} \QD(X_0).
$$
The Hitchin's fibration \begin{equation}\label{eq-pi}
\pi:\M\rightarrow\B
\end{equation}
takes a Higgs bundle $(E,\theta)$ to the determinant of $\theta$.
The following classical result was proved by Hitchin \cite{hitchin}.
\begin{theorem}\label{thm-hitchin}
The map $\pi$ is proper, surjective,  and is an abelian fibration over the open set
\[\B'=\left\{\phi\in\B : \, \phi\text{ has only simple zeros}\right\}.\]
Denote by $\M'=\pi^{-1}(\B')$. Then $\M'\subset\M^s$.  
\end{theorem}

\subsection{The determinant map} Let  $X\in \Teich$. Each $A\in \sl_2(\Omega^1(X))$ is of the form
 $$
 A=\begin{bmatrix}
            \alpha & \beta\\
            \gamma & -\alpha
        \end{bmatrix}\ 
$$        
for some  $\alpha,\beta,\gamma\in \Omega^1(X)$. The map
$$
\dett:\sl_2(\Omega^1(X)) \sslash \SL_2(\C)\rightarrow \QD(X)
$$
is given by $\dett(A)=-\alpha^2-\beta\gamma$.
\begin{definition} Let $X\in \Teich$. We say that $W\subset \QD(X)$ is a dilation invariant subset if 
$$
\phi \in W\quad \implies \quad t\phi \in W, \,\, \text{for every $t>0$}.
$$
\end{definition}
The following lemma follows from the first principles.  It is proved in Section \ref{section-appendix}.  
\begin{lemma}\label{lemma-elem} There is a Zariski open and dense, and dilation invariant, subset 
$\B^{\ur} \subset \B$ such that the  map
\begin{equation}\label{eq-det-0}
\dett:\sl_2(\Omega^1(X_0)) \sslash \SL_2(\C)\rightarrow \B
\end{equation}
is a finite unramified covering map over $\B^{\ur}$. 
\end{lemma}

\begin{definition}\label{definition-''}  Set
\[\B''=\B^{\ur}\cap\B'=\left\{\phi\in\B^{\ur}:\,\phi\text{ has only simple zeros}\right\}.\]
\end{definition}
\begin{lemma}\label{lemma-B''}
The set $\B''$ is a dilation invariant,  Zariski open and dense, subset of $\B$.
\end{lemma}
\begin{proof} By construction both $\B'$ and $\B^{\ur}$ are dilation invariant,  Zariski open and dense, subsets of $\B$. Thus, $\B''$ is as well.
\end{proof}

\subsection{Higgs fields with large determinant} 

\begin{definition} For a dilation invariant subset $W\subset\B$ we define
$$
W(t)=\left\{\phi\in W:\, \|\phi\|>t\right\},
$$
where  $\|\phi\|=\int_{X_0} |\phi|$.
\end{definition}
For a subset $U\subset \Teich$, we let 
$$
\Syst_U=\left\{(X,A)\in \Syst : \, X\in U \right\}. 
$$
The following result underpins the proof of Theorem \ref{thm-main}
\begin{theorem}\label{thm-3} Suppose $W\subset \B''$ be a closed dilation invariant sector. Let  $U$ be an open neighbourhood of $X_0$ in $\T$. Then there exists $t>0$ such that $F(\Syst_U)$ contains $\pi^{-1}(W(t))$.
\end{theorem}
\begin{remark} When we say that $W\subset \B''$ is a closed subset we mean that it is a subset of $\B''$ which is closed as a subset of $\B\setminus \{0\}$.
\end{remark}

\subsection{The strategy}

Theorem \ref{thm-main} follows from Theorem \ref{thm-3}. Namely, given any discrete faithful representation 
$\sigma:\pi_1(\Sigma)\to \SLC$, one can precompose $\sigma$ by a suitable element of the mapping class group of $\Sigma$ so that the resulting representation falls into our favourite region  $\pi^{-1}(W(t))$. To do this we use  the results by Daskalopoulos-Dostoglou-Wentworth  in \cite{d-d-w} which relate the Morgan-Shalen compactification of  $\Rep$ with the  compactification of $\M$.

We briefly explain the strategy for proving Theorem \ref{thm-3}. We need to show that the image $F(\S^s)$ contains a large part of the moduli space $\M$. Good properties of the Hitchin's fibration  $\pi$ allow us to do this in steps. In the first step we show that $\pi\circ F$ contains a large part of $\B$, while in the second step we show that the corresponding fibres of $\pi$ are contained in $F(\S^s)$. 
\vskip .1cm
\noindent
\textbf{Step 1: Asymptotics of $\pi\circ F$.} A  key building block in our analysis is the result that the map 
$(\pi\circ F):\Syst^s \to \B$ is well approximated by the (much simpler) determinant  map. Namely, we show that for $A \in \sl_2(\Omega^1(X))$, with  $\dett(A)\neq0$, the equality 
$$
(\pi\circ F)(X,tA))=\frac{t^2}{4}\dett(A) (1+o(1)), \quad t \to \infty
$$ 
holds.
We prove this  using WKB analysis to estimate the monodromy around loops. The Hitchin WKB and Riemann-Hilbert WKB problems were studied by Katzarkov-Noll-Pandit-Simpson \cite{k-n-p-s}, Gaiotto-Moore-Neitzke \cite{g-m-n},  Mochizuki  \cite{mochizuki},  and others. Using this we  deduce the following weak version of Theorem \ref{thm-3}.

\begin{theorem}\label{thm-3 weak} Suppose $W\subset \B''$ is a closed  dilation invariant sector. Then there exists $t>0$ such that $(\pi\circ F)(\Syst_{X_{0}})$ contains $W(t)$. 
\end{theorem}
\vskip .1cm
\noindent
\textbf{Step 2: Surjectivity onto the fibres.} By Step 1 we already know that the image of $F$ contains at least one point on each fibre $\pi^{-1}(\phi)$ for every $\phi\in W(t)$. We use the continuity method to show that the image contains the whole fibre. This is motivated by the fact that $F$ is generically a local diffeomorphism. The key point here is that the fibres $\pi^{-1}(\phi)$  have a uniformly bounded diameters when $\phi\in W(1)$. This readily follows from the work of  Mazzeo-Swoboda-Weiss-Witt \cite{m-s-w-w} where they prove that the Hitchin metric $g_{L^{2}}$, and the semiflat metric $g_\sf$, are asymptotic to each other along rays $\M_{t\phi}$.

We connect two points on the same fibre by a path $\alpha$ of bounded length, and try to lift $\alpha$ to a curve $\beta$ in $\S$ so that $F\circ\beta=\alpha$. As always, in order to apply the continuity method the main difficulty is an a priori estimate. Here we need to show that $\beta$ is contained in a compact subset of $\S$. We prove this by estimating the derivative of $F$ from below, especially in the horizontal direction in $\Syst$ (varying the complex structure in $\Teich$).

A useful observation here is that $F^*\xi$ looks like ${\dett}^*\eta_{\can}$, where $\xi$ and $\eta_{\can}$ are, respectively, the Goldman symplectic form on the character variety, and the canonical symplectic form on the cotangent bundle $\QDD\cong T^*\Teich$ 
(here $\QDD$ is the bundle of holomorphic quadratic differentials over $\Teich$). Therefore a lower bound of $dF$ in the horizontal direction is the same as an upper bound of $dF$ in the fibre direction. So it remains to give an upper bound of $dF$ in the fibre direction. We achieve this using Yau's Schwarz lemma.

\subsection{Organisation of the paper}
In Section 2 we give the proof of Theorem \ref{thm-main} using Theorem \ref{thm-3} and the results 
from \cite{d-d-w} relating the Morgan-Shalen compactification of  $\Rep$ with the  compactification of $\M$. In Section 3 we state and prove some preliminary results mostly concerning the determinant map $\dett:\Syst\to \QDD$.

In Section 4 we recall  the results of Mochizuki \cite{b-d-h-h}, \cite{mochizuki}, showing that the map 
$(\pi\circ \F)$ is well approximated by the  determinant  map (we also sketch the proofs). In Section 5, using standard arguments from Teichm\"uller dynamics, we construct suitable $WKB$-curves required for the application of Mochizuki's results.   Using this, we prove Theorem \ref{thm-3 weak} in Section 6 and Section 7. 

The pullbacks of the Goldman symplectic form from $\Rep$ to $\Syst$, and $\M$, respectively, are computed in Section 8  (these results are  well known). Then in Section 9 we state 
Proposition \ref{proposition: bound on fiber} and Lemma \ref{lemma: diam bound}. The proposition provides an upper bound on the norm $\|dF\|_{g_L^{2}}$ in the fibre direction, where $g_{L^{2}}$ is the Hitchin metric on $\M'$.
The lemma provides the uniform upper bound on the diameter of the fibre $\pi^{-1}(\phi)$ when $\phi\in \B''$.
Combining these with the computations involving the Goldman symplectic form enables us to complete the proof Theorem \ref{thm-3} in Section 10.

Finally, Section 11 and Section 12 are devoted to proving Proposition \ref{proposition: bound on fiber} and Lemma \ref{lemma: diam bound}. The proof of Lemma \ref{lemma: diam bound} relies on the fact that the metric $g_{L^{2}}$  is asymptotic to the semiflat metric $g_\sf$ along  $\pi^{-1}(t\phi)$. On the other hand, Proposition \ref{proposition: bound on fiber} is proved using the Schwartz lemma with respect to a suitably constructed metric of negative holomorphic sectional curvature which we construct in a neighbourhood of $\pi^{-1}(t\phi)$.

\subsection*{Acknowledgement} The authors would like to thank the following people for discussions:  Yitwah Cheung (ergodic theory), Yi Huang (Thurston's compactification), and Mao Sheng (Hitchin's moduli space).

\section{The mapping class group actions} 
The goal of this section is to prove Theorem \ref{thm-main} assuming  Theorem \ref{thm-3}. 
To ensure that the assumptions in  Theorem \ref{thm-3} are satisfied  we apply the following well known result: The accumulation set of the mapping class group orbit of a discrete faithful representation 
$\sigma\in \Rep^{\df}$ is equal to the subset of  the Morgan-Shalen compactification of  
$\Rep^{\df}$ consisting of ''small actions". This is the only place in the proof of Theorem \ref{thm-main} where we require that $\sigma$ is a discrete faithful representation.

\subsection{Mapping class group actions} Let $\Mod(\Sigma)$ denote the mapping class group of $\Sigma$. Then $\Mod(\Sigma)$ acts  on both $\Syst$ and  $\Rep$, and the map $\RH$ is equivariant with respect to these  actions. Namely, let  $T\in \Mod(\Sigma)$. Then the  automorphism $T:\Teich\to \Teich$  yields the automorphism $T:\Syst\to \Syst$ given by $(X,A)\to (T(X),A)$ (note that $T(X)$ and $X$ project to the same point in the moduli space).

On the other hand,  the outer automorphism 
$T:\pi_1(\Sigma)\to \pi_1(\Sigma)$ yields the action $T:\Rep\to \Rep$ given by $\rho\to \rho \circ T$. 
The map $\RH$ is equivariant with respect to these actions, that is, 
$$
\RH(X,A)\circ T=\RH(T(X),A).
$$

\subsection{$\Mod(\Sigma)$-orbit} Denote by $\Rep^{\df}\subset \Rep$ the subset of discrete faithful representations, and let $\sigma\in\Rep^{\df}$. By $\Orb(\sigma)$ we denote the $\Mod(\Sigma)$-orbit of $\sigma$, that is, 
$$
\Orb(\sigma)=\{\sigma\circ T:\, T\in \Mod(\Sigma)\}.
$$
\begin{lemma}\label{lemma-sacekaj} Suppose that 
\begin{equation}\label{eq-orb}
\Orb(\sigma) \cap \RH(\S_U)\ne \emptyset
\end{equation} 
for every neighbourhood $U\subset \Teich$ of $X_0$. Then Theorem \ref{thm-main} holds.
\end{lemma}
\begin{proof} If (\ref{eq-orb}) holds then  there exists $T\in \Mod(\Sigma)$ such that $\sigma \circ T \in\RH(\S_U)$ which  means that  $\sigma \in\RH(\S_{T^{-1}(U)})$. This shows that there exists $(X,A)\in \S$ such that $\RH(X,A)=\sigma$, and that $X$ is very close to $X_0$ as an unmarked Riemann surface (although as marked surfaces $X$ and $X_0$ are potentially far away in $\Teich$). Since this holds for every $U$, and every non-hyperelliptic $X_0$, we see that Theorem \ref{thm-main} holds.
\end{proof}

\subsection{Proof of Theorem \ref{thm-main}}  The following lemma is proved at the end of this section.

\begin{lemma}\label{lemma-ta} Let $W\subset\B$ be a nonempty, open, dilation invariant sector. 
Then 
\begin{equation}\label{eq-orb-1}
\NAH(\Orb(\sigma))\cap \pi^{-1}(W(t)) \ne \emptyset
\end{equation}
for every $t>0$.
\end{lemma}
But first, we complete the proof of Theorem \ref{thm-main}. It remains to prove that (\ref{eq-orb}) holds for every neighbourhood $U\subset \Teich$ of $X_0$. Fix such a neighbourhood $U$.

Let $W\subset \B''$ be a non-empty, open, dilation invariant, sector such that $\overline{W}\subset \B''$ (such $W$ exists by Lemma \ref{lemma-B''}).  By Theorem \ref{thm-3} there exists $t>0$ such that $\pi^{-1}(W(t))\subset F(\Syst_U)$. Combining this with (\ref{eq-orb-1}) we get that 
$$
\NAH(\Orb(\sigma)) \cap F(\Syst_U) \ne \emptyset. 
$$
Replacing $\RH=\NAH^{-1}\circ F$ in the previous equality proves (\ref{eq-orb}).

\subsection{Compactifying $\Rep$ and $\M$}  Let $\overline{\Rep}$ be the Morgan-Shalen compactification of $\Rep$, and  let $\overline{\Rep^{\df}}\subset \overline{\Rep}$ be the induced compactification of  $\Rep^{\df}$. The boundary $\partial{\Rep}$ is identified with the projective classes of length functions on $\pi_1(\Sigma)$. The boundary $\partial{\Rep^{\df}}\subset \partial{\Rep}$ consists of the so called small actions. It follows from a theorem of Skora (as observed  in \cite{d-d-w}) that $\partial{\Rep^{\df}}$ is homeomorphic to the space of projective measured foliations.

The mapping class group $\Mod(\Sigma)$ acts continuously on $\overline{\Rep}$ keeping 
$\overline{\Rep^{\df}}$ invariant. The above identification between $\partial{\Rep^{\df}}$ with the space of projective measured foliations respects the action of $\Mod(\Sigma)$.  Since the action of 
$\Mod(\Sigma)$ is minimal on the space of projective measured foliations it follows that it is minimal on $\partial{\Rep^{\df}}$ (this means that the $\Mod(\Sigma)$-orbit of every point in $\partial{\Rep^{\df}}$ is dense in $\partial{\Rep^{\df}}$).  In turn, this implies that for any $\sigma\in \Rep^{df}$ the accumulation set of   $\Orb(\sigma)$  contains the entire boundary $\partial{\Rep^{\df}}$ (in fact, the accumulation set is equal to $\partial{\Rep^{\df}}$).

On the other hand, in \cite{d-d-w} Daskalopoulos-Dostoglou-Wentworth defined the compactification 
$\overline{\M}=\M\cup\partial\M$, where $\partial\M \cong \partial {\wh{\B}}$, and $\wh{\B}=\{\phi\in\B:\,\|\phi\|<1\}$. The topology on $\overline{\M}$ is given by the map $\wh{\pi}:\M\to \wh{\B}$ which is obtained by composing the Hitchin fibration $\pi:\M\to \B$ with the map $\B\to \wh{\B}$  given by
$$
\phi \to \frac{4\phi}{1+4||\phi||}, \quad\quad \phi\in \B.
$$

\begin{lemma}\label{lemma-senke} The accumulation set of $(\wh{\pi}\circ\NAH)(\Orb(\sigma))$ contains the entire boundary  $\partial{\wh{\B}}$.
\end{lemma}
\begin{proof} By the main theorem of \cite{d-d-w} the homeomorphism $\NAH$ extends to a continuous surjective map $\overline{\NAH}:\overline{\Rep}\rightarrow\overline{\M}$. The restriction 
$\overline{\NAH}:\overline{\Rep^{\df}}\to \overline{\M}$ is a homeomorphism onto its image, and the restriction
$\overline{\NAH}:\partial{\Rep^{\df}}\to \partial{\M}$ is a homeomorphism.

Above we showed that  the accumulation set of $\Orb(\sigma)$  contains the entire boundary $\partial{\Rep^{\df}}$. 
Since $\overline{\NAH}:\partial{\Rep^{\df}}\to \partial{\M}$ is a homeomorphism, it follows that the accumulation set of $\NAH(\Orb(\sigma))$  contains $\partial{\M}$. Composing this with $\wh{\pi}$ proves the lemma.
\end{proof}

\subsection{Proof of Lemma \ref{lemma-ta}} 

Let $\wh{W}=\{\phi \in W:\,\|\phi\|<1\}$. Then $\partial{\wh{W}}$ is a nonempty open subset of 
$\partial{\wh{\B}}$. Set 
$$
\wh{W}_s=\wh{W}\cap W(s)=\{\phi\in W: \, s<||\phi||<1 \}.
$$
From Lemma \ref{lemma-senke} it follows that   $(\wh{\pi}\circ\NAH)(\Orb(\sigma))\cap \wh{W}_s\ne \emptyset$ for every $0<s<1$. This is the same as saying  that $(\pi\circ\NAH)(\Orb(\sigma))\cap W(t)\ne \emptyset$, where $s=\frac{4t}{1+4t}$. Taking $\pi^{-1}$ of both the left and the right hand side yields the proof of Lemma \ref{lemma-ta}.

\section{Preliminaries on the determinant map}\label{section-appendix} 

By $\QDD$ we denote the bundle of holomorphic quadratic differentials over the Teichm\"uller space 
$\Teich$. In this section we establish some properties of the  determinant map 
$$
\dett:\Syst \to \QDD
$$
given by $\dett(X,A)=\dett(A)$.  In particular, we prove Lemma \ref{lemma-elem}.

\subsection{Two standard results}
We begin by proving a claim and a lemma (both of which are pretty routine).

\begin{claim}\label{claim-r} Suppose $X\in \Teich$ is not hyperelliptic (unless the genus of $X$ is equal to two in which case $X$ is  arbitrary). Then the image of the map $\dett:\sl_2(\Omega^1(X)) \sslash \SL_2(\C)\rightarrow\QD(X)$  is not contained in a proper Zariski closed subset of $\QD(X)$.
\end{claim}    
\begin{proof}   It suffices to show that the image of $\dett$ contains an open subset of $\QD(X)$ (in the manifold topology). By the implicit function theorem it suffices to show that the differential $d(\dett)$  is surjective at at least one point. 
 
 Note that $\dett:\sl_2(\Omega^1(X))\rightarrow\QD(X)$ is given by
        \[\begin{bmatrix}
            \alpha & \beta\\
            \gamma & -\alpha
        \end{bmatrix}\,\mapsto\, (-\alpha^2-\beta\gamma)\]
        where $\alpha,\beta,\gamma\in \Omega^1(X)$.
        Its differential at $\begin{bmatrix}
            \alpha & \beta\\
            \gamma & -\alpha
        \end{bmatrix}$, in the direction $\Phi=\begin{bmatrix}
            \varphi_1 & \varphi_2\\
            \varphi_3 & -\varphi_1
        \end{bmatrix}$, where $\varphi_1,\varphi_2,\varphi_3\in \Omega^1(X)$, is given by
        \[d(\dett)\left(\begin{bmatrix}
            \varphi_1 & \varphi_2\\
            \varphi_3 & -\varphi_1
        \end{bmatrix}\right)= -\tr(A\Phi)  =-2\alpha\varphi_1-\gamma\varphi_2-\beta\varphi_3.\]     
   It follows that $d(\dett)$ is surjective if and only if      
\begin{equation}\label{eq-max}
\left\{\alpha\varphi_1+\beta\varphi_2+\gamma\varphi_3:\, \varphi_1,\varphi_2,\varphi_3\in\Omega^1(X)\right\}=\QD(X).
\end{equation}
By the classical theorem of Max Noether (see Theorem 1.1 in \cite{gieseker}, or Theorem 4 in \cite{m-t}) there exist $\alpha,\beta,\gamma\in\Omega^1(X)$ such that  (\ref{eq-max}) holds.      
This implies that $d(\dett)$ is surjective at at least one point,  proving the claim.
\end{proof}

The standard result by Chevalley (see Exercise 3.22 in Hartshorne \cite{hartshorne}) states that
given a dominant morphism $f:X\rightarrow Y$  between (integral, quasi-projective) varieties of the same dimension, the pre-image of a generic point in $Y$ is finite. Combining this with the Zariski's Main Theorem yields the following lemma.

\begin{lemma}\label{lemma-r-1}  Let $f:X\to Y$ be a morphism between (integral, quasi-projective) varieties of the same dimension. Suppose that $f$ is dominant (i.e. the image of $f$ is not contained in a proper Zariski closed subset of $Y$). Then there is a Zariski open dense subset $U\subset Y$ such that $f$ is finite unramified over $U$.
\end{lemma}

\subsection{Covering properties of $\dett$ and proof of Lemma \ref{lemma-elem} }

For clarity, in this subsection we let $\dett_X$ denote the restriction of $\dett$ to  $\sl_2(\Omega^1(X))/\!/\SL_2(\C)$.
We also let  $\Teich^{\ur}\subset \Teich$ denote the locus of non-hyperelliptic Riemann surfaces (unless the genus is two in which case we let $\Teich^{\ur}=\Teich$).
The following is an  immediate corollary of Claim \ref{claim-r} and Lemma \ref{lemma-r-1}.

\begin{lemma}\label{lemma-r-2} Suppose $X\in \Teich^\ur$. There is a Zariski open dense subset $Z_X\subset \QD(X)$ such that $\dett_X:\sl_2(\Omega^1(X))/\!/\SL_2(\C)\rightarrow\QD(X)$ is a finite unramified covering map over $Z_X$. 
\end{lemma}
If $\dett_X:\sl_2(\Omega^1(X))/\!/\SL_2(\C)\rightarrow\QD(X)$ is a finite unramified covering map over a set $Z\subset \QD(X)$ then the same is true for the set $\{t \phi:\, \phi \in  Z, \, t \in \C\setminus \{0\}\}$. Thus, we may assume that $Z_X$ from Lemma \ref{lemma-r-2} is dilation invariant.
In particular, letting  $\B^{\ur}=Z_{X_{0}}$ proves Lemma \ref{lemma-elem}.

Set
$$
\QDD^\ur=\bigcup_{X\in \Teich^\ur}\, Z_X.
$$
\begin{lemma}\label{lemma-r-3} Consider  the determinant map $\dett:\Syst \to \QDD$. Then every $\phi\in \QDD^\ur$ has neighbourhood $U\subset \QDD$ such that $\dett$ is a finite covering over $U$.
\end{lemma}

\begin{proof} Suppose that $\phi\in \QD(X)$. Then every preimage of $\phi$ under the map $\dett$ belongs to 
$\sl_2(\Omega^1(X))/\!/\SL_2(\C)$, that is, $\dett^{-1}(\phi)=\dett^{-1}_X(\phi)$. Therefore, it follows from Lemma \ref{lemma-r-2} that $\dett^{-1}(\phi)$ consists of finitely many points.

Suppose $(X,A)\in \dett^{-1}(\phi)$. We can conclude from Lemma \ref{lemma-r-2} that the the derivative of $\dett$ is a bijection between the corresponding tangent spaces. By the implicit function theorem, there exists a neighbourhood $V\subset \Syst$ of $(X,A)$ so that the restriction of $\dett$ to $V$ is a diffeomoprhism onto its image. Define $U$ as the intersection of (finitely many) neighbourhoods $\dett(V)$ of $\phi$. Then $\dett$ is a finite cover over $U$. \end{proof}

\subsection{$F$ is a local diffeomorphism near $\B^{\ur}$ }
We have  the following lemma.

\begin{lemma}\label{lemma-elem-1} There exists a dilation invariant neighbourhood $V\subset \Syst$ of the set  $\dett^{-1}(\B^{\ur})$ such that $F$ is a local diffeomorphism on $V$.
\end{lemma}

\begin{proof}
Suppose
$(X_0,A)\in \dett^{-1}(\B^{\ur})$, where  
 \[A=\begin{bmatrix}
            \alpha & \beta\\
            \gamma & -\alpha
        \end{bmatrix}\]
        for some $\alpha,\beta,\gamma\in \Omega^1(X_0)$. Since $d(\dett)$ is surjective at $A$, it follows that the equality (\ref{eq-max}) holds. Thus, the equality (\ref{eq-max}) holds at each point of $\dett^{-1}(\B^{\ur})$. 
        
Since the condition    (\ref{eq-max})  is dilation invariant, there exists a dilation invariant  neighbourhood $V\subset \Syst$ of $\dett^{-1}(\B^{\ur})$ so that (\ref{eq-max}) holds at every point of $V$.  By Theorem 1 in \cite{m-t} this means that $d(\RH)$ is bijective    at each point of $V$. \end{proof}

\section{WKB analysis}\label{section-WKB}

The next three sections are aimed at proving Theorem \ref{thm-3 weak}. The proof relies on computing the  projective length functions of the representations  $\RH(X,tA)$ and $\NAH^{-1}(E,t\theta)$ when $t\to \infty$.
This is the content of Lemma \ref{lemma: WKB-0} and Lemma \ref{lemma: WKB-1} below which essentially follow from the results of Mochizuki in \cite{b-d-h-h} and  \cite{mochizuki}.

\subsection{The two lemmas}

Let $\phi\in \B$ be a  holomorphic quadratic differential. Given a curve  $\gamma\subset X_0$ we define its width by
$$
w_\phi(\gamma)=\int_\gamma |\Real\big(\sqrt{\phi(z)} \, dz\big)|.
$$
We say that $\gamma$ is a WKB-curve with respect to $\phi$  if $\gamma$ is transverse to the vertical foliation of $\phi$ (in particular, such $\gamma$ stays away from the zeroes of $\phi$).

Let $\rho\in \Rep$, and suppose that $\gamma\subset \Sigma$ is a closed curve. Considering 
$\gamma$ as a conjugacy class in $\pi_1(\Sigma)$,  we let $\chi_\rho(\gamma)$  denote the character of $\rho(\gamma)$ (that is, $\chi_\rho(\gamma)$ is the trace of the conjugacy class  $\rho(\gamma)$).

\begin{lemma}\label{lemma: WKB-0} 
Let $(X_n,A_n)\in \Syst$ be a sequence converging to some $(X_0,A_0)$ with $\dett (A_0) \neq 0$.
Suppose $t_n>0$ is sequence such that $t_n\to \infty$ when $n\to \infty$, and let $\rho_n=\RH(X_n,t_nA_n)$.
Then for any closed curve  $\gamma\subset X_0$, we have
$$
 \limsup_{n\to\infty}\frac{1}{t_n}\log|\chi_{\rho_{n}}(\gamma)|\leq w_{\dett (A_0)}(\gamma).
 $$
 Moreover, if $\gamma$ is a WKB-curve with respect to $\dett (A_0)$ then the limit exists and the equality holds.
\end{lemma} 
 \begin{lemma}\label{lemma: WKB-1} Let $(E_n,\theta_n)\in \M$ be a sequence such that 
 $\dett(\theta_n)$ converges to some $\phi_0\in \B$ with $\phi_0\ne 0$. Suppose $t_n>0$  is such that $t_n\to \infty$ when $n\to \infty$, and let $\rho_n=\NAH^{-1}(E_n,t_n\theta_n)$.
Then for any closed curve  $\gamma\subset X_0$, we have
$$
 \limsup_{n\to\infty}\frac{1}{2t_n}\log|\chi_{\rho_{n}}(\gamma)|\leq w_{\phi_{0}}(\gamma).
 $$
Moreover, if $\gamma$ is a WKB-curve with respect to $\phi_0$ then the limit exists and the equality holds.
\end{lemma}  
The case when $\gamma$ is a WKB-curve in Lemma \ref{lemma: WKB-0} basically follows from Proposition A.2 in \cite{b-d-h-h}, and in Lemma \ref{lemma: WKB-1} from Theorem 1.5 of \cite{mochizuki}. To treat the case when $\gamma$ is a not a WKB-curve one can for example cut  $\gamma$ into WKB-segments and then apply the same results on each such segment, or use the argument we present below which is probably simpler. For the sake of completeness we sketch the proofs of the two lemmas in the remainder of this section (the reader interested in the application of these  lemmas may initially skip the rest of this section).

\subsection{Proof of Lemma \ref{lemma: WKB-0}} By perturbing $\gamma$, we may assume that 
$\gamma$ does not pass through any zero of $\dett(A_0)$. To simplify the notation we write $(X,A)$ for $(X_n,A_n)$,  $t$ for $t_n$, and $\rho$ for $\rho_n$.

Note that $A\in\sl_2(\Omega^1(X))\cong H^0(X,K_X\otimes\End_0(\O_X^{\oplus2}))$ is a form valued trace zero endomorphism of $\O_X^{\oplus2}$ (the trivial rank two bundle over $X$). Locally, and away from the zeros of $\dett (A)$, we have a decomposition $\O_X^{\oplus2}=V_+\oplus V_-$ into eigenspaces of $A$. Via $\gamma:[0,1]\rightarrow X$ this pulls back to a globally defined decomposition $\gamma^*\O_X^{\oplus2}=\gamma^*V_+\oplus\gamma^*V_-$. 

Take $C^\infty$ basis $v_{\pm}(s)$ of $\gamma^*V_{\pm}$, and express $\gamma^*(d+tA)$ in terms of this basis. We get
    \[\gamma^*(d+tA)=d+(B(s)+tA(s))ds,\]
    where $B(s)$ is uniformly bounded with respect to $t$, and
    \begin{align*}
        A(s)=
       \begin{bmatrix}
            a(s) & 0\\
            0 & -a(s)
       \end{bmatrix},\quad\pm a(s)ds=\gamma^*(\pm\sqrt{\dett (A)}).
    \end{align*}
    Consider the ODE
    \begin{align*}
        \begin{cases}
        &\frac{d}{ds}f_t(s)+D(s)f_t(s)=0,\\
        &f_t(0)=I_2.
        \end{cases}
    \end{align*}
    where $D(s)=B(s)+tA(s)$.
    We have
    \[\chi_{\rho}(\gamma)=\tr\left(Tf_t(1)\right),\]
    where $T$ is the transition matrix from $v_{\pm}(1)$ to $v_{\pm}(0)$. By standard ODE theory, we have
    \[\|f_t(1)\|\leq \exp\left(\int_0^1\frac{1}{2}\|D(s)+D(s)^*\|ds\right).\]
    Note that the right hand side is dominated by
    \[\exp\left(\int_0^1\frac{1}{2}\|B(s)+B(s)^*\|ds\right)\exp\left(t\int_0^1|\re (a(s))|ds\right),\]
    and that
    \[\int_0^1|\re(a(s))|ds=w_{\dett (A)}(\gamma).\]
    Putting these together we get
    \[\log|\chi_{\rho}(\gamma)|\leq t\cdot w_{\dett (A)}(\gamma)+O(1).\]
    The first assertion of Lemma \ref{lemma: WKB-0} follows. As for the second assertion, $\gamma$ is a WKB-curve meaning that $\re(a(s))\neq0$ for all $s\in[0,1]$. Possibly by exchanging $V_+$ and $V_-$ we get $\re(a(s))>0$ for all $s\in[0,1]$. In particular, $a(0)$ and $a(1)$ having the same sign means that the transition matrix $T$ is diagonal. By standard ODE theory, we have
    \begin{align*}
        f_t(1)=e^{-\int_0^1b_{11}(s)+ta(s)ds}
        \left(\begin{bmatrix}
            1&0\\
            0&0
        \end{bmatrix}+o(1)\right),\,\quad t\rightarrow\infty.
    \end{align*}
    This implies
    \[\log|\chi_{\rho}(\gamma)|= t\cdot w_{\dett (A)}(\gamma)+O(1).\]
    The second assertion follows.

\subsection{Proof of Lemma \ref{lemma: WKB-1}}
    As above we may assume that $\gamma$ does not pass through any zero of $\phi_0$ and to simplify the notation we write $(E,\theta)$ for $(E_n,\theta_n)$, $t$ for $t_n$, and $\rho$ for $\rho_n$. Locally, away from the zeros of $\dett(\theta)$, we have a decomposition $E=E_+\oplus E_-$ into eigenspaces of $\theta$. Let $h_t$ be the harmonic metric on $(E,t\theta)$. By Theorem 1.4 in \cite{mochizuki} we have
    \begin{align}\label{eq7}
        |D_{h_t}-D_{h_{t,0}}|_{h_{t,0}}\leq C,
    \end{align}
    where $h_{t,0}$ is a Hermitian metric approximating $h_t$ (away from the zeros of $\dett(\theta)$). Furthermore, $h_{t,0}$ has the properties that the decomposition $E=E_+\oplus E_-$ is orthogonal with respect to $h_{t,0}$, and that $h_{t,0}|_{E_{\pm}}=h_t|_{E_\pm}$. Moreover, 
    \[D_{h_{t,0}}=d_{h_{t,0}}+t\theta+t{\theta}_{h_{t,0}}^*,\]
    where $d_{h_{t,0}}=\partial_{h_{t,0}}+\bar\partial$ is the Chern connection of $h_{t,0}$. Pulling back via $\gamma$, and trivialising $\gamma^*E_{\pm}$ by sections $e_{\pm}(s)$ which are parallel with respect to $d_{h_{t,0}}$, we get
    \[\gamma^*D_{h_t}=d+(B(s)+tA(s))ds.\]
    Here $B(s)$ is uniformly bounded with respect to $t$ (by (\ref{eq7})), and
    \begin{align*}
        A(s)=
       \begin{bmatrix}
            a(s) & 0\\
            0 & -a(s)
       \end{bmatrix},\quad\pm a(s)ds=\pm\gamma^*\left(\sqrt{\dett (\theta)}+\overline{\sqrt{\dett(\theta)}}\right).
    \end{align*}
    Note that now the transition matrix $T$ is unitary with respect to $h_{t,0}$. The rest of the proof is the same as in the previous subsection.

\section{Constructing a WKB-curve} To be able to simultaneously apply Lemma \ref{lemma: WKB-0} and
Lemma \ref{lemma: WKB-1}, we need to be able to find a closed WKB-curve whose widths  distinguish between different quadratic differentials.
\begin{lemma}\label{lemma: finding curve} Let $\phi,\psi\in \B$ be such that $\phi\ne \psi$, and 
$\|\phi\|\geq\|\psi\|$. Then there exists a closed curve $\gamma\subset X_0$ which is a WKB-curve with respect to  $\phi$, and which satisfies the strict inequality
$$
w_\psi(\gamma)<w_\phi(\gamma).
$$
\end{lemma}
The idea is to use the Birkhoff ergodic theorem and construct a suitable simple closed curve $\gamma$  which is nearly horizontal with respect to  $\phi$. Without loss of generality we may assume $\|\phi\|=1$. Thus, we have $\|\psi\|\le 1$. We prove the lemma in the remainder of this section.

\subsection{The probability measure} Define the probability measure $\mu$ on $X_0$ by letting $d\mu=|\phi|$.
For $\theta\in[0,2\pi)$, we define the function $f_\theta$ on $X_0$ by
$$
f_\theta(p)=\left|\re\sqrt{\frac{\psi}{e^{i\theta}\phi}}\right|(p), \quad\quad p\in X_0.
$$
Then $f_\theta\in L^1(\mu)$. Since $\psi\ne \phi$ it follows that for almost every $p\in X_0$ we have the strict inequality
$$
f_0(p)=\left|\re\sqrt{\frac{\psi}{\phi}}\right|(p)<\frac{|\psi|}{|\phi|}(p).
$$
This yields the strict inequality 
\begin{equation}\label{eq-cos-00}
\int_{X_{0}} f_0\, d\mu= \int_{X_{0}}\left|\re\sqrt{\frac{\psi}{\phi}}\right|\cdot|\phi| < \|\psi\| \le 
1.
\end{equation}
\begin{claim}\label{claim-cos} There exists $\theta_0>0$
such that the strict inequality
\begin{equation}\label{eq-cos}
\int_{X_{0}} f_\theta\, d\mu < \cos(\theta_0) 
\end{equation}
holds for every $\theta\in [0,\theta_0]$.
\end{claim} 
\begin{proof} Since $f_\theta\to f_0$ in $L^1(\mu)$, and using (\ref{eq-cos-00}), we can find $q<1$, and $\theta_1>0$, such that 
$$
\int_{X_{0}} f_\theta\, d\mu\le q
$$
for every $\theta\in [0,\theta_1]$. By choosing $0<\theta_0 \le \theta_1$ small enough we can arrange that $q \le \cos(\theta_0)$ which proves the claim.
\end{proof}

\subsection{The Birkhoff Ergodic Theorem} As shown by Kerckhoff-Masur-Smillie in \cite{k-m-s}, for almost all $\theta\in[0,2\pi)$ the horizontal foliation of $e^{i\theta}\phi$ is uniquely ergodic. Fix such $\theta\in [0,\theta_0]$. Assuming $p\in X_0$ is such that the horizontal trajectory of the quadratic differential $e^{i\theta}\phi$ which contains $p$ is non-singular, we let $\alpha_p:\R \to X_0$ denote the isometric parametrisation of this trajectory which is uniquely defined by the condition $\alpha_p(0)=p$.

Since the horizontal flow (with respect to $e^{i\theta}\phi$)  is ergodic on $X_0$, and since $f_\theta\in L^1(\mu)$, it  follows from the Birkhoff Ergodic Theorem for flows  that
$$
\lim_{T\to \infty} \frac{1}{T}\int_0^T (f_\theta\circ \alpha_p)(t)\,dt=
\int_{X_{0}} f_\theta\, d\mu
$$
for almost every $p\in X_0$.  Combining this with (\ref{eq-cos}) yields the following claim.

\begin{claim}\label{claim-birk-2} There exist $p\in X_0$, $\theta\in [0,\theta_0]$, and $T_0>0$, such that  the inequality
\begin{equation}\label{eq-birk} 
\frac{1}{T}\int_0^T (f_\theta\circ \alpha_p)(t)\,dt \le \cos(\theta_0)
\end{equation}
holds for every $T\ge T_0$.
\end{claim}

\subsection{Proof of Lemma \ref{lemma: finding curve}}
Fix $p\in X_0$, and $\theta\in [0,\theta_0]$,  from Claim \ref{claim-birk-2}. Let $\alpha^T_p$ be the arc $\alpha_p([0,T])$. 
Since $\alpha^T_p$ is a horizontal arc for $e^{i\theta}\phi$, it follows that for every $T>0$ we have the equalities
$$
w_\psi(\alpha^T_p)=\int_0^T (f_\theta\circ \alpha_p)(t)\,dt,\quad\quad w_\phi(\alpha^T_p)=T\cos(\theta/2).
$$
Combining this with (\ref{eq-birk}) yields that for every $T\ge T_0$ we have 
$$
w_\psi(\alpha^T_p)\le T\cos(\theta_0)=w_\phi(\alpha^T_p)+ T(\cos(\theta_0)-\cos(\theta/2)).
$$
Let $\delta_0=T_0(\cos(\theta_0/2)-\cos(\theta_0))>0$. Then, for every $T\ge T_0$ we have
\begin{equation}\label{eq-birk-1} 
w_\psi(\alpha^T_p)\le w_\phi(\alpha^T_p)-\delta_0.
\end{equation}

Since the horizontal flow is uniquely ergodic for $e^{i\theta}\phi$, it follows that there exists and increasing sequence $T_n$ of times  such that $\alpha_p(T_n)\to p$ when $n\to \infty$. Moreover, we can choose $T_n$ so that one can close up the arc $\alpha^{T_{n}}_p$ by inserting a very short, nearly horizontal, interval such that the new closed curve $\gamma_n$ has the following properties  
$$
|w_\psi(\gamma_n)-w_\psi(\alpha^T_p)|\le \frac{\delta_0}{3},\quad \quad  
|w_\phi(\alpha^T_p)-w_\phi(\gamma_n)| \le \frac{\delta_0}{3}.
$$
Replacing this in (\ref{eq-birk-1}) we get that 
$$
w_\psi(\gamma_n)\le w_\phi(\gamma_n)-\frac{\delta_0}{3}.
$$
Letting $\gamma=\gamma_n$ proves the lemma.

\section{Asymptotic of $\pi\circ F$}

The purpose of this section is to prove Proposition \ref{proposition: WKB} which is derived from the WKB analysis discussed above. This  proposition is  used in the proofs of Theorem \ref{thm-3 weak} and Theorem \ref{thm-3}.

Recall that  $\QDD$  denotes the fibre bundle of holomorphic quadratic differentials over $\Teich$, and  let $Q:\QDD \to \Teich$ denote the corresponding fibration. Choose a (smooth) trivialisation
\begin{equation}\label{eq-QL}
(Q,L):\QDD\to \Teich \times \B.
\end{equation}
Also, by $d_\Teich$ we denote the Teichm\"uller metric on $\Teich$.

\begin{proposition}\label{proposition: WKB} Suppose that  $W\subset \B''$ is a 
closed and dilation invariant subset. Then there exist 
\begin{itemize}
\item[(a)]  a dilation invariant neighbourhood $V\subset \QDD$ of $W$, 
\vskip .1cm
\item[(b)] a function $\delta:[0,\infty) \times (0,\infty) \to [0,\infty)$, with  $\delta(s,t)\to 0$, when both $s\to 0$ and $t\to \infty$, so that for every $(X,A)\in \dett^{-1}(V(t))$ we have
$$
\left\|4(\pi\circ F)(X,A)- L\big(\dett (A)\big) \right\|\le  \,\delta\big(d_\Teich(X,X_0),t\big) \|L\big(\dett (A)\big)\|.
$$
\end{itemize}
\end{proposition}
\begin{proof}  We argue by contradiction. 

\begin{lemma}\label{lemma-contra}
Suppose that the proposition does not hold. 
Then  exist $A_0\in \slo(\Omega^1(X_0))$, with $\|\dett(A_0)\|=1$, and 
sequences $X_n\in \Teich$, $A_n \in \slo(\Omega^1(X_n))$, and $t_n>0$,  satisfying the equalities 
$$
\lim_{n\to \infty} X_n=X_0, \quad \lim_{n\to \infty} A_n=A_0, \quad \lim_{n\to \infty} t_n=\infty,
$$
such that
\begin{equation}\label{eq30}
\left\|\frac{4}{t_n^2}(\pi\circ F)(X_n,t_nA_n)-\dett (A_n)\right\|>q
\end{equation}
for every $n\in \N$, and some  $q>0$.
\end{lemma}
\begin{proof}
Since we assume that the proposition does not hold, there exist 
\begin{itemize}
\item[(a)] a decreasing sequence of  dilation invariant neighbourhoods $V_n\subset \QDD$ of $W$, with $\bigcap_{1}^{\infty} V_n=W$,
\vskip .1cm
\item[(b)] a  sequence of positive numbers $s_n\to \infty$, 
\vskip .1cm 
\item[(c)] elements $(X_n,B_n)\in \dett^{-1}(V_n(s_n))$,
\vskip .1cm 
\item[(d)] $q>0$,
\end{itemize}
such that 
\begin{equation}\label{eq30-1}
\left\|4(\pi\circ F)(X_n,B_n)-L\big(\dett (B_n)\big)\right\|>q\|L\big(\dett (B_n)\big)\|.
\end{equation}
First of all, we conclude  $\lim_{n\to \infty} X_n=X_0$. 
Secondly, let $t^2_n=\|\dett(B_n)\|$. Then  $t_n\to \infty$. Set $A_n=t^{-1}_nB_n$. Then $A_n\in \dett^{-1}(V_n)$, and $\|\dett(A_n)\|=1$. Thus, there exists $\psi_0\in W$, with $\|\psi_0\|=1$, so that after passing to a subsequence if necessary, we have $L\big(\dett(A_n)\big) \to \psi_0$.  

Since $X_n\to X_0$ it follows that $X_n\in \Teich^\ur$ for large $n$ (meaning that $X_n$ is not hyperelliptic, or the genus of $X_n$ is two). 
By Lemma \ref{lemma-r-3}  the sequence of fibres $\dett^{-1}(\dett(A_n))$ converges to the fibre 
$\dett^{-1}(\dett(\psi_0))$. It follows that there is $A_0\in \dett^{-1}(\dett(\psi_0))$ so that $\lim_{n\to \infty} A_n=A_0$ (again after passing to a subsequence if necessary).
Clearly, $\|\dett(A_0)\|=1$. 
We can rewrite (\ref{eq30-1}) as 
$$
\left\|\frac{4}{t_n^2}(\pi\circ F)(X_n,A_n)-L\big(\dett (A_n)\big)\right\|>q,
$$
which proves the lemma.
\end{proof}

\begin{definition} Let $\phi_n\in \QD(X_n)$, and  $s_n \geq 0 $, be such that $\|\phi_n\|=\|\dett(A_0)\|=1$, and
$$
(\pi\circ F)(X_n,t_nA_n)=s_n^2\phi_n.
$$
\end{definition}

The proof of the proposition follows from the following two claims.

\begin{claim}\label{claim-film} 
$$
\lim_{n\to \infty} \phi_n=\dett (A_0). 
$$
\end{claim}
\begin{proof}
Assume that the claim does not hold. Then, possibly after passing to a subsequence, we have $\phi_n\to \phi_0\neq\dett (A_0)$. Then one of the following two equalities hold
\begin{equation}\label{eq-ass}
\liminf_{n\to\infty}\frac{t_n}{2s_n}\leq 1,
\end{equation}
or
\begin{equation}\label{eq-ass-1}
\liminf_{n\to\infty}\frac{2s_n}{t_n}\leq 1
\end{equation}
We assume that (\ref{eq-ass}) holds. The other case is handled analogously.

By construction, we have $\|\phi_0\|=\| \dett(A_0)\|$.   
Therefore, Lemma \ref{lemma: finding curve} guarantees the existence of a closed  WKB-curve $\gamma\subset X_0$ (with respect to $\phi_0$)  such that
 \begin{align}\label{eq12}
 w_{\phi_0}(\gamma)>w_{\dett (A_0)}(\gamma).
 \end{align}
Set $\rho_n=\RH(X_n,t_nA_n)$. Then it follows from Lemma \ref{lemma: WKB-0}  that
    \begin{equation}\label{eq13}
        \limsup_{n\rightarrow\infty}\frac{1}{t_n}\log|\chi_{\rho_{n}}(\gamma)|\leq w_{\dett (A_0)}(\gamma).
    \end{equation}

Let $F(X_n,t_nA_n)=(E_n,s_n\theta_n)$. Then $\phi_n=\dett(\theta_n)$. By Lemma \ref{lemma: WKB-1} we have
    \begin{align}\label{eq14}
        \lim_{n\to\infty}\frac{1}{2s_n}\log|\chi_{\rho_{n}}(\gamma)|= w_{\phi_0}(\gamma).
    \end{align}
Putting together (\ref{eq12}), (\ref{eq13}), and (\ref{eq14}), gives
$$
 \lim_{n\to\infty}\frac{1}{2s_n}\log|\chi_{\rho_{n}}(\gamma)|=w_{\phi_0}(\gamma)>w_{\dett (A_0)}(\gamma)\ge\lim_{n\to\infty}\frac{1}{t_n}\log|\chi_{\rho_{n}}(\gamma)|. 
 $$
But this contradicts (\ref{eq-ass}) and we are done \end{proof}

\begin{claim}\label{claim-film-1}

$$
\lim_{n\to \infty} \frac{2s_n}{t_n}=1. 
$$
\end{claim}

\begin{proof} As in the previous proof we let $\rho_n=\RH(X_n,t_nA_n)$, and  $F(X_n,t_nA_n)=(E_n,s_n\theta_n)$. Then $\phi_n=\dett(\theta_n)$, and as we saw in the previous claim $\phi_n\to \dett(A_0)$ when $n\to \infty$.

    Take a closed WKB-curve $\gamma$ (with respect to $\dett (A_0)$). By Lemma \ref{lemma: WKB-0} we have
    \begin{align}\label{eq33}
        \lim_{n\to\infty}\frac{1}{t_n}\log|\chi_{\rho_{n}}(\gamma)|=w_{\dett (A_0)}(\gamma).
    \end{align}
     Likewise, by Lemma \ref{lemma: WKB-1}  we have
    \begin{align}\label{eq34}
        \lim_{n\to\infty}\frac{1}{2s_n}\log|\chi_{\rho_{n}}(\gamma)|= w_{\dett (A_0)}(\gamma).
    \end{align}
  Combining (\ref{eq33}) and (\ref{eq34}) proves the claim. \end{proof} 
  We can now complete the proof of the proposition. Since 
  $$
  \frac{1}{s_n^2}(\pi\circ F)(X_n,t_nA_n)=\phi_n,
  $$
  and applying the previous two claims  implies that
    \[\lim_{n\to \infty} \frac{4}{t_n^2}(\pi\circ F)(X_n,t_nA_n)=\dett (A_0).\]
 This contradicts  (\ref{eq30}) and we are done.
\end{proof}

\section{The surjectivity of the map $(\pi\circ F)$} 
The purpose of this section is to  prove the following stronger version of Theorem \ref{thm-3 weak}.

\begin{theorem}\label{thm-3 weak-1} Suppose $W\subset \B''$ is a closed dilation invariant sector. Let $Z\subset \B$ be a dilation invariant neighbourhood of $W$. Then there exists $t>0$ such that 
$(\pi\circ F)\big(\dett^{-1}(Z)\big)$ contains $W(t)$. 
\end{theorem}

Note that $\dett^{-1}(Z)\subset \S_{X_{0}}$, hence Theorem \ref{thm-3 weak-1} implies Theorem \ref{thm-3 weak}.

\subsection{An auxiliary claim} 

We need  the following elementary auxiliary claim. The proof is left to the reader (compare with  Lemma 5.1 in \cite{c-m}). 

\begin{claim}\label{claim-aux} 
Let $B_1$ denote the open unit ball in $\R^n$.  Suppose
$$
f_t:\overline{B_1} \to \R^n
$$ 
is  family of maps which is continuous  in $p \in \overline{B_1}$,  and  such that 
$$
\lim_{t\to \infty} f_t(p)=p,\quad \quad \text{uniformly in   $p \in \overline{B}_1$}.
$$
Then for every compact subset $K\subset B_1$ there exists $t_0=t_0(K)$ such that for every $t\ge t_0$, the slice  $f_t(\overline{B_1})$ contains $K$.
\end{claim}

\subsection{Proof of Theorem \ref{thm-3 weak-1}}
Let $\phi_0\in W$ with $\|\phi_0\|=1$. Given $r>0$, we let
$$
B_r(\phi_0)=\{ \phi\in \B:\, \|\phi-\phi_0\|<r\}.
$$
By Lemma \ref{lemma-elem} we can find a small enough $r>0$ such  that  $\dett$ is an unbranched cover over the closed ball 
$\overline{B_{r}(\phi_0)}$. Furthermore, by reducing $r$ if necessary, we may assume $B_r(\phi_0)\subset Z$. Thus, there exists a closed topological ball $P \subset \dett^{-1}(Z)$ so that the restriction  $\dett:P\to \overline{B_r(\phi_0)}$  is a homeomorphism.

Since $P$ is a compact set, by Proposition \ref{proposition: WKB} we have that 
\begin{equation}\label{eq-ch}
\lim_{t\to \infty}\left\|\frac{4}{t^2}(\pi\circ F)(X_0,tA)-\dett (A) \right\|=0
\end{equation}
uniformly for every $(X_0,A)\in P$. Define the functions $g_t:P\to \B$ by letting
$$
g_t(X_0,A)=\frac{4}{t^2}(\pi\circ F)(X_0,tA),
$$
and let $H:\R^n \to \B$ be any homeomorphism which maps  $\overline{B_1}$ onto $\overline{B_r(\phi_0)}$.
Define the maps $f_t:\overline{B_1}\to \R^n$ by
$$
f_t=H^{-1}\circ g_t \circ \dett^{-1}\circ H.
$$

Observe that by (\ref{eq-ch}), $f_t$ satisfies the assumptions in Claim \ref{claim-aux}. Let $U_{\phi_{0}}\subset \B$ be a small enough neighbourhood of $\phi_0$ so that  $K=H^{-1}(\overline{U_{\phi_{0}}})$ is a compact subset of $B_1$. It follows that for some $t_1=t_1(\phi_{0})>0$ the inclusion  $K\subset f_t(\overline{B_1})$ holds for every $t\ge t_1$. Unpacking this shows that  $tU_{\phi_{0}} \subset (\pi\circ F)(P) \subset (\pi\circ F)(\dett^{-1}(Z))$. Since $W$ is closed and dilation invariant we can find $t_0>0$ so that $tU_{\phi}\subset  (\pi\circ F)(\dett^{-1}(Z))$ for every  $\phi\in W$ with $\|\phi\|=1$, and every  $t\ge t_0$. This proves the theorem.

\section{The symplectic structures on $\Syst$, $\Rep$, and $\M$ }
The goal of this section is to compute the pull-backs of the Goldman symplectic form on $\Rep^s$ to   
$\Syst^s$, and $\M^s$, respectively. 

\subsection{The Goldman symplectic form}

Let $\rho\in \Rep^s$. Recall that the tangent space $T_\rho\Rep$ coincides with $H^1(\pi_1(\Sigma),\Ad_\rho)$, where $\Ad_\rho$ is the $\pi_1(\Sigma)$-module with underlying vector space $\sl_2(\C)$ equipped with the adjoint action. 
Denote by $\mathcal{E}$ the rank 2 local system on $\Sigma$ arising from the representation $\rho$, and observe the de Rham isomorphism 
$$
\iota:H^1(\Sigma,\sl(\mathcal{E}))\to H^1(\pi_1(\Sigma),\Ad_\rho) \cong T_\rho \Rep.
$$
Then, the Goldman symplectic form \cite{goldman} on $\Rep$ is given by
\begin{align}\label{eq: Goldman}
\xi(A,B)=\int_{\Sigma}\tr(A\wedge B),
\end{align}
where $A,B\in H^1(\Sigma,\sl(\mathcal{E}))\cong T_\rho \Rep$.

\subsection{The tangent space to $\Teich$} 
We let $\Belt(X)$ denote the space of Beltrami differentials on a Riemann surface  $X \in \Teich$.
Recall the usual pairing
$$
\left<\mu,\phi\right>=\int_X \mu \phi
$$
where $\mu\in \Belt(X)$, and $\phi \in \QD(X)$. Define the equivalence relation on $\Belt(X)$ by letting $\mu \sim \nu$ if $\left<\mu,\phi\right>=\left<\nu,\phi\right>$ for every $\phi\in \QD(X)$. We have the standard identification
$$
\left\{[\mu]:\, \mu \in \Belt(X) \right\} \cong T_X\Teich.
$$ 
The Teichm\"uller norm of the tangent vector $[\mu]$ is given by 
\begin{equation}\label{eq-fins}
\|[\mu]\|_*=\max_{\phi\in \QD(X)\setminus 0} \, \frac{\left<\mu,\phi\right>}{\|\phi\|}.
\end{equation}

\subsection{The tangent space to $\Syst$} 
Fix $(X,A)\in \Syst$. Following \cite{m-t}, we  define the vector space
$$
Z(X,A)=\{ (\mu,\dot{A}) : \mu \in \Belt(X), \,\, \dot{A}^{0,1}=\mu A \},
$$
where $\dot{A}$  is a closed $\slc$-valued 1-form on $X$. The linear equivalence relation on $Z(X,A)$ is given by letting
$(\mu,\dot{A}) \sim (\nu,\dot{B})$ if the following two conditions are satisfied
\begin{enumerate} 
\item $[\mu]=[\nu]$,
\vskip .1cm
\item $\dot{A}^{1,0}-\dot{B}^{1,0}-\partial{T} \in [A,\slc]$, for some function  $T:X\to \slc$.
\end{enumerate}
Denote by $[(\mu,\dot{A})]$ the corresponding equivalence class. Then by \cite{m-t} we have
$$
\{ [(\mu,\dot{A})] : \mu \in \Belt(X), \, \dot{A}^{0,1}=\mu A \} \cong T_{(X,A)}\Syst.
$$

\subsection{The derivative of the Riemann-Hilbert map}
Suppose that $\rho=\RH(X,A)$.  Then $\mathrm{Ad}_\rho$ is  the monodromy of the flat connection 
$d_A=d+\mathrm{ad}(A)$. Let $(\mu,\dot{A})\in T_{(X,A)}\S$. Then $\dot{A}$ is a $d_A$-closed $\sl_2(\C)$-valued 1-form, and hence it defines a cohomology class $\chi\in H^1(X,\sl(\mathcal{E}))$. We have (see Theorem 2 in \cite{m-t})
\begin{equation}\label{eq-dRH}
d(\RH)\big([(\mu,\dot{A})]\big)=-\iota(\chi).
\end{equation}

\subsection{The  pullback of the Goldman form to  $\Syst$}

\begin{proposition}\label{proposition: symplectic}
Let $(X,A)\in\S$. Consider tangent vectors $v,w\in T_{(X,A)}\S$ of the form $v=(\mu,\dot{A}_1)$ and $w=(0,\dot{A}_2)$. Then
$$
(\RH^*\xi)(v,w)=-\left<\mu, d(\dett )(\dot{A}_2) \right>,
$$
where $d(\dett )$ is the derivative of the map 
$\dett:\sl_2(\Omega^1(X)) \sslash \SLC\to \QD(X)$ at the point $A\in \sl_2(\Omega^1(X)) \sslash \SLC$.
\end{proposition}

\begin{proof}
By (\ref{eq: Goldman}) we have
$$
(\RH^*\xi)(v,w)=\xi(d(\RH)(\mu,\dot{A}_1),d(\RH)(0,\dot{A}_2))=\int_X \tr(\dot{A}_1\wedge\dot{A}_2).
$$
We have $\dot{A}_2^{0,1}=0$, so $\dot{A}_2$ is of  $(1,0)$-type. Since $\dot{A}_2$ is closed, it is in fact holomorphic and thus $\dot{A}_2\in\sl_2(\Omega^1(X))$. Combining this with the equality  $\dot{A}_1^{0,1}=\mu A$, yields
$$
\int_X\tr(\dot{A}_1\wedge\dot{A}_2)=\int_X\tr(\dot{A}_1^{0,1}\wedge\dot{A}_2)=\int_X\tr(\mu A\wedge\dot{A}_2)=\int_X\mu\,\tr(A\dot{A}_2).
$$
So we have
\begin{equation}\label{eq-golda}
(\RH^*\xi)(v,w)=\int_X\mu\,\tr(A\dot{A}_2).
\end{equation}
On the other hand, we compute
$$
d(\dett)(\dot{A}_2)=\frac{d}{dt}\det(A+t\dot{A}_2)=-\tr(A\dot{A_2}).
$$
Replacing this into (\ref{eq-golda}) proves the proposition.
\end{proof}

\subsection{The hyperk\"ahler structure on $\M$}
The moduli space $\M^s$ carries a natural hyperk\"ahler structure $(g_{L^2},I,J,K)$, where $g_{L^2}$
is the Hitchin $L^2$-metric, and $I,J,K$ the three complex structures satisfying the quaternionic relations ($I^2 = J^2 = K^2 = IJK = -1$). 
The complex structure $I$ is the natural complex structure as the moduli space of Higgs bundles, while $J$ is identified with the natural complex structure on the $\SLC$-character variety.

The corresponding real-valued K\"ahler forms are given by
$$
 \omega_I(\cdot, \cdot) = g_{L^2}(I\cdot, \cdot), \quad \omega_J(\cdot, \cdot) = g_{L^2}(J\cdot, \cdot), \quad \quad \omega_K(\cdot, \cdot) = g_{L^2}(K\cdot, \cdot).
$$
These forms are closed and non-degenerate, making the space $\M^s$ K\"ahler with respect to each of these complex structures. For each choice of complex structure, the remaining two K\"ahler forms can be packaged into a holomorphic symplectic form which is a closed $(2,0)$-form with respect to that specific structure. Specifically
$$
\eta_I = \omega_J + i \omega_K ,\quad \eta_J = \omega_K + i\omega_I\quad \eta_K = \omega_I + i\omega_J.
$$
Hitchin \cite{hitchin} computed the pullback of the Goldman holomorphic symplectic form to $\M^s$ via $\NAH^{-1}$.
\begin{proposition}\label{proposition-symplectic} We have $(\NAH^{-1})^*\xi=-i\eta_J$. In particular,
$$
(\NAH^{-1})^*\re(\xi)=\omega_I.
$$
\end{proposition}

\section{Bounds on the derivative of $F$}
The purpose of this section is to prepare the proof of Theorem \ref{thm-3} which is given in the next section. Consider the fibration $P:\Syst\to \Teich$ given by $P(X,A)=X$.
The proof of Theorem \ref{thm-3} rests on the following  result which  estimates the derivative of $F$ in the  direction of the fibre of $P$.
\begin{proposition}\label{proposition: bound on fiber}
Let  $W\subset\B''$ be a closed  dilation invariant sector. Then  there exist a dilation invariant neighbourhood $V\subset \QDD$ of $W$, and a constant $C>0$, such that for every $(X,A)\in\dett^{-1}(V)$, and every $[(0,\dot{A})] \in T_{(X,A)}\S_X$, we have
    \begin{equation}\label{eq: bound on fiber}
            \|dF([(0,\dot{A})])\|_{g_{L^2}}\leq \frac{C}{\sqrt{\|\dett (A)\|}}\,\|d(\dett)(\dot{A})\|.
    \end{equation}
\end{proposition}
\vskip .3cm
We also need the following lemma which is an immediate consequence of results in \cite{m-s-w-w}.
\begin{lemma}\label{lemma: diam bound} Let $W\subset \B'$ be any closed dilation invariant sector.  Then there is a constant $D=D(W)>0$ such that for all $\phi\in W(1)$ the inequality
$$
\mathrm{diam}_{g_{L^2}}(\M_\phi)\leq D.
$$
holds, where $\M_\phi=\pi^{-1}(\phi)$.
\end{lemma}

\subsection{The lower bound on $dF$ in the horizontal direction} Combining Proposition \ref{proposition: bound on fiber} with the previous results regarding the Goldman symplectic form enables us to obtain a lower bound on the derivative of $F$ in the horizontal direction of the fibration $P$.

\begin{lemma}\label{lemma: horizontal lower bound} Let  $W\subset\B''$ be a closed  dilation invariant sector. Then  there exist a dilation invariant neighbourhood $V\subset \QDD$ of $W$, and a constant $C>0$, such that for every $(X,A)\in \dett^{-1}(V)$, and every $v=[(\mu,\dot{A})] \in T_{(X,A)}\S$, we have
\begin{equation}\label{eq-laba}
\sqrt{\|\dett (A)\|} \, \|[\mu]\|_{*} \le C\|dF([(\mu,\dot{A})])\|_{g_{L^2}}.
\end{equation}
\end{lemma}
\begin{proof} Let $A \in \dett^{-1}(W)\cap \big(\sl_2(\Omega^1(X_0))\sslash \SLC\big)$. Then the map 
$\dett:\sl_2(\Omega^1(X_0))\sslash  \SLC \to \B$
is a local diffeomorphism near $A$ because $W\subset \B''$. 
Thus, we can choose a dilation invariant neighbourhood $V_1\subset \QDD$ of $W$ so that the map 
$\dett:\sl_2(\Omega^1(X)) \sslash \SLC\to \QD(X)$ is local diffeomorphism near any $A\in \dett^{-1}(V_1)\cap \big(\sl_2(\Omega^1(X))\sslash \SLC\big)$.

On the other hand, we let $V_2\subset \QDD$ be a  dilation invariant neighbourhood of $W$ from
Proposition \ref{proposition: bound on fiber} so that (\ref{eq: bound on fiber}) holds. Set $V=V_1\cap V_2$, and fix $(X,A)\in \dett^{-1}(V)$.

By $\phi\in \QD(X)$ we denote the unit norm quadratic differential such that
\begin{equation}\label{eq-ext}
\left<\mu,\phi\right>=\|[\mu]\|_{^*}.
\end{equation}
Since $(X,A)\in \dett^{-1}(V_1)$, there exists $\dot{B}\in \sl_2(\Omega^1(X))$ such that 
$$
d(\dett)(\dot{B})=-\phi.
$$ 
Furthermore, since  $\dot{B}^{0,1}=0$, we have that 
$$
w=[(0,\dot{B})]\in T_{(X,A)}\S
$$ 
is a well defined tangent vector.  
By Proposition \ref{proposition: symplectic} and Proposition \ref{proposition-symplectic} we have 
$-i(F^*\eta_J)(v,w)=(\RH^*\xi)(v,w)$ which yields
$$
-i(F^*\eta_J)(v,w)=-\left<\mu,d(\dett)(\dot{B})\right>=-\left<\mu,-\phi\right>=\|[\mu]\|_*. 
$$
Thus
\begin{equation}\label{eq53}
-i\eta_J(dF(v),dF(w))=\|[\mu]\|_*.
\end{equation}

Recall that for tangent vectors $u_1,u_2$ we have $\re\big(-i\eta_J(u_1,u_2)\big)=\omega_I(u_1,u_2)=
g_{L^2}(Iu_1,u_2)$. Since $I$ is unitary with respect to $g_{L^2}$, we obtain the estimate
$$
|\re\big(-i\eta_J(u_1,u_2)\big)|\leq\|u_1\|_{g_{L^2}}\|u_2\|_{g_{L^2}}. 
$$
So (\ref{eq53}) yields
    \begin{align}\label{eq54}
    \|[\mu]\|_*\le     \|dF(v)\|_{g_{L^2}}\|dF(w)\|_{g_{L^2}}.
    \end{align}
    On the other hand by Proposition \ref{proposition: bound on fiber} we have
    \begin{align}\label{eq55}
        \|dF(w)\|_{g_{L^2}}\leq\frac{C}{\sqrt{\|\dett A\|}}\|d(\dett)(\dot{B})\|.
    \end{align}
    where $C>0$ is the constant from Proposition \ref{proposition: bound on fiber}.
But $\|d(\dett)(\dot{B})\|=\|\phi\|=1$, so combining (\ref{eq54}) and (\ref{eq55})  proves the lemma.
\end{proof}

\section{Proof of theorem \ref{thm-3}}

We  collate  previous results in the following lemma.

\begin{lemma}\label{lemma-okolina} Suppose we are given the following:
\begin{itemize}
\item[(a)] a closed  dilation invariant sector $W\subset\B''$, 
\vskip .1cm
\item[(b)] a neighbourhood $U \subset \Teich$ of $X_0$.
\end{itemize}  
Then  there exist a dilation invariant neighbourhood $V\subset \QDD$ of $W$,  and a constant $t_0>0$, with the following properties
\begin{itemize}
\item[(1)] the conclusions of Proposition  \ref{proposition: WKB}  and Lemma \ref{lemma: horizontal lower bound} hold on $V$,
\vskip .1cm
\item[(2)] for every $t\ge t_0$, and for  every $\phi \in W(t)$, there exists $(X_0,A_0)\in \dett^{-1}(V\cap \B)$ such that 
$(\pi\circ F)(X_0,A_0)=\phi$,
\vskip .1cm
\item[(3)] $F$ is a local diffeomorphism on $\dett^{-1}(V)$, and $\dett^{-1}(V)\subset \S_U$.
\end{itemize} 
\end{lemma}
\begin{proof} By $V_0\subset \QDD$ we denote the neighbourhood of $W$ from Proposition  \ref{proposition: WKB}, and
by $V_1\subset \QDD$  the neighbourhood of $W$ from Lemma \ref{lemma: horizontal lower bound}.
By Lemma \ref{lemma-elem-1} there exists a neighbourhood $V_2\subset \QDD$  of $W$ such that  $F$ is a local diffeomorphism on $\dett^{-1}(V_2)$. 
Set $V=V_0\cap V_1\cap V_2 \cap \dett(\Syst_U)$ 
(note that $\dett(\S_U)$ is the union of spaces $\QD(X)$ where $X\in U$). Then (1) and  (3) hold for this choice of $V$. 
Since $V\cap \B$ is dilation invariant neighbourhood of $W$, from 
Theorem \ref{thm-3 weak-1}  we conclude that (2) holds as well for some $t_0>0$. The lemma is proved.
\end{proof}

\subsection{The endgame} 

Suppose   $W\subset \B''$ is a closed dilation invariant sector. Let $V\subset \QDD$ be its  neighbourhood  from Lemma \ref{lemma-okolina}. Let $t>0$, and choose  $\phi \in W(t)$. Set  $\M_{\phi}=\pi^{-1}(\phi)$.
To prove Theorem \ref{thm-3}, it suffices to show that  $\M_{\phi} \subset F(\dett^{-1}(V))$ when $t$ is large enough.

 Assuming $t$ is large enough,  by (2) from Lemma \ref{lemma-okolina} there is $(X_0,A_0) \in \dett^{-1}(V\cap \B)$ such that $q_0=F(X_0,A_0)\in \M_{\phi}$. Let $q\in\M_{\phi}$ be any other point. We need to show $q\in F\big(\dett^{-1}(V)\big)$. Let $D_0$ be the constant from Lemma \ref{lemma: diam bound}. Then there exist  
\begin{itemize}
\item a unit speed (with respect to $g_{L^{2}}$) smooth path $\alpha:[0,s_0]\rightarrow\M_{\phi}$,
\vskip .1cm
\item  $\alpha(0)=q_0,\ \alpha(s_0)=q$,
\vskip .1cm
\item  the length of $\alpha$ is at most $D_0$, that is, we have $s_0\le D_0$.
\end{itemize}

\begin{definition} 
We define the subset $S\subset[0,s_0]$ consisting of $s\in[0,s_0]$ such that there exists a path 
$\beta_s:[0,s]\to \dett^{-1}(V)$ so that $F\circ\beta_s=\alpha$ on $[0,s]$. Let $s^*=\sup S$.
\end{definition}
Suppose $s_1,s_2\in S$.  By $(3)$ from Lemma \ref{lemma-okolina} we know  that $F$ is a local  diffeomorphism on $\dett^{-1}(V)$. This implies that the paths $\beta_{s_{1}}$ and $\beta_{s_{2}}$ agree on $[0,s_1]\cap [0,s_2]$. Thus, there exists the (unique) path $\beta:[0,s^*)\to \dett^{-1}(V)$ so that $F\circ\beta_s=\alpha$ on $[0,s^*)$. Let $(X(s),A(s))=\beta(s)$, for $s\in [0,s^*)$.

\begin{claim}\label{claim-jedan} The set $S$ is a non-empty open subset of $[0,s_0]$. 
\end{claim}
\begin{proof}
Clearly $0\in S$ (thus, $S$ is non-empty). Suppose $s\in S$. Since $F$ is a local  diffeomorphism on $\dett^{-1}(V)$, and since $\beta(s)\in  \dett^{-1}(V)$, we can continue the path on some neighbourhood of $s$ which proves the claim. 
\end{proof}
To prove Theorem \ref{thm-3} we need to show that $S$ is also a closed subset of $[0,s_0]$. It suffices to prove the following lemma.

\begin{lemma}\label{lemma-dosta} There exists $t_0$ so that for each $t\ge t_0$, and for each $\phi\in W(t)$, there exists a neighbourhood $V_1\subset \QDD$ of $\phi$ such that $\overline{V}_1$ is a compact subset of $V$, and so 
that $(X(s),A(s))\in \dett^{-1}(V_1)$ for every $s\in [0,s^*)$.
\end{lemma}

\begin{proof}
We begin with the  claim. 
\begin{claim}\label{claim-dosta-1} There exists $C>0$ such that 
$$
d_\Teich(X(s),X_0)\le  \frac{C}{\sqrt{t}}
$$
for every $s\in [0,s^*)$, and every $t>0$.
\end{claim}
\begin{proof} Differentiating $\beta:[0,s^*)\to \dett^{-1}(V)$, we define $[(\mu(s),\dot{A}(s))]   \in T_{(X(s),A(s))}$ by
$$
d\beta\left(\frac{\partial}{\partial{s}}\right)=[(\mu(s),\dot{A}(s))].
$$
Let $C_1$ be the constant from  Lemma \ref{lemma: horizontal lower bound}. Then (since the conclusion of this lemma holds on $V$) we have
$$
\sqrt{t} \, \|[\mu(s)]\|_{*} \le C_1\|dF([(\mu(s),\dot{A}(s))])\|_{g_{L^2}},
$$
for $s\in [0,s^*)$.  But,  $F\circ \beta:[0,s^*)\to \M_{\phi}$ is infinitesimally isometric so
$$
\|dF([(\mu(s),\dot{A}(s))])\|_{g_{L^2}}=1, \quad \forall s\in [0,s^*).
$$
We get
$$
\|[\mu(s)]\|_{*} \le \frac{C_1}{\sqrt{t}}, \quad \forall s\in [0,s^*),
$$
which in turn proves the claim by letting $C=s_0C_{1} \le D_0C_1$.
\end{proof}

Since the conclusion of  Proposition  \ref{proposition: WKB} holds on $V$, we conclude that for large enough $t$, and every $s\in [0,s^*)$, the inequality
$$
\left\|4\phi- L\big(\dett (A(s))\big) \right\|\le  \,\delta\big(d_\Teich(X(s),X_0),t\big) \|L\big(\dett (A(s))\big)\|
$$
holds, where $\delta:[0,\infty)\times (0,\infty) \to [0,\infty)$ is a function such that $\delta(s,t)\to 0$ when $s\to 0$ and $t\to \infty$. 
By the triangle inequality we have
$$
\|L\big(\dett (A(s))\big)\|\le \frac{4\|\phi\|}{1-\delta\big(d_\Teich(X(s),X_0),t\big)}.
$$
Combining the previous two inequalities  with  Claim \ref{claim-dosta-1} yields
\begin{equation}\label{eq-zavrs}
\left\|4\phi- L\big(\dett (A(s))\big) \right\|\le  \,
\frac{4\delta\left(\frac{C}{\sqrt{t}},t\right)}{1-\delta\left(\frac{C}{\sqrt{t}},t \right)} \|\phi\|,
\end{equation}
for every $s\in [0,s^*)$. We are ready to construct the required neighbourhood $V_1$. 

\begin{definition} Recall that $B_r(\phi)=\{\psi\in \B:\, \|\psi-\phi\|<r\}$.  We also let $B_r(X_0)\subset \Teich$ denote the ball of radius  $r$ centred at $X_0$ (with respect to the Teichm\"uller metric).
\end{definition}

Since $V$ is dilation invariant, there exists $\epsilon>0$ so that $B_{\epsilon \|\phi \|}(\phi)$ is compactly contained in $V$ for every $\phi\in W$.
Let $V_1(\epsilon)=(Q,L)^{-1}\big(B_\epsilon(X_0) \times B_{\epsilon \|\phi \|}(\phi)   \big)$, where  $(Q,L):\QDD\to \Teich \times \B$ is the trivialisation defined above (see (\ref{eq-QL})). Then  $\overline{V}_1(\epsilon)$ is a compact subset of $V$. On the other hand, from Claim \ref{claim-dosta-1} and (\ref{eq-zavrs}), it follows that when $t$ is large enough we have that $(X(s),A(s))\in V_1(\epsilon)$ for every $s\in [0,s^*)$. This proves the lemma.
\end{proof}

\section{Bounding $dF$ in the fibre direction}

In this section we prove Proposition \ref{proposition: bound on fiber} assuming  Proposition \ref{proposition: curvature} which is stated below. As we shall see below, it is natural  to use the Schwarz lemma in order to establish an upper bound on the derivative of $dF$. However, we remind the reader that $\pi\circ F$ is not a holomorphic map. Indeed, $F$ is holomorphic with respect to the complex structure $J$ on $\M^s$, while $\pi$ is holomorphic with respect to the complex structure $I$ on $\M^s$. 

Let $\phi_0\in \B$, and suppose $r>0$ is small enough so that $B_{r}(\phi_0)\subset \B''$. 
Thanks to the  results by  Mazzeo-Swoboda-Weiss-Witt \cite{m-s-w-w}
on the asymptotic geometry of the Hitchin moduli space we know that
the metric $g_{L^2}$ is asymptotic to the semiflat metric $g_{\sf}$. Therefore, the restriction of hyperk\"ahler structure to 
$\pi^{-1}(B_{tr}(t\phi_0))$ is modelled on a flat product $B\times T$, where $B$ is a ball and $T$ is a flat torus. 
\begin{remark}
An important feature of the semiflat metric is that it is conic on the base. So for $t$ large, $\pi^{-1}(B_{rt}(t\phi_0))$ is modelled on $B\times T$ with $\mathrm{diam} (B)\sim \sqrt{t}$. This explains the factor $\sqrt{\|\dett(A)\|}$ in (\ref{eq: bound on fiber}).
\end{remark}

We construct a metric $g_J$ on $\pi^{-1}(B_{tr}(t\phi_0))$ which is K\"ahler with respect to $J$, and   which has negative holomorphic sectional curvature. In turn, this will enable us to apply Yau's Schwarz lemma \cite{yau}
to bound the derivative of $F$. The following proposition will be proved in the next section.

\begin{proposition}\label{proposition: curvature} Let $Z\subset \B''$ be a closed subset consisting of norm one quadratic differentials. There are constants $\kappa, C, r, t_0>0$  such that for each $\phi_0\in Z$, and every $t\ge t_0$, there is a metric $g_J$ on $\pi^{-1}(B_{rt}(t\phi_0))$ with the following properties
\begin{itemize}
\item[(a)]  $g_J$ is K\"ahler with respect to $J$,
\vskip .1cm
\item[(b)] $g_J$  has holomorphic sectional curvature $<-\kappa$,
\vskip .1cm
\item[(c)]  $C^{-1}\|v\|_{g_{L^{2}}}\leq \sqrt{t}\cdot \|v\|_{g_{J}}\leq C\|v\|_{g_{L^{2}}}$, for every tangent vector $v$.
\end{itemize}
\end{proposition}

\subsection{The Yau-Schwartz lemma} 

Denote by $\mathbb{B}_R^n\subset\C^n$ the  ball of radius $R$ centred at the origin. The following is a version of the Yau-Schwartz lemma \cite{yau} proved by Royden in \cite{royden}.

\begin{lemma}\label{lemma-yau} Let $\kappa>0$, and suppose that $(M,g)$ is a K\"ahler manifold with holomorphic sectional curvature $<-\kappa$. Then there exists a constant $C>0$ so that for every  holomorphic map  $h:\mathbb{B}_R^n\rightarrow M$ we have 
 $$
 \|dh(v)\|_g \leq \frac{C}{R}\|v\|
 $$
where $v\in T_0\mathbb{B}_R^n$ is a tangent vector of Euclidean norm $\|v\|$, and where $\|dh(v)\|_g $ is the $g$-norm of the vector $dh(v)\in T_{h(0)}M$.
\end{lemma}
\begin{remark}
The analogous statement holds if we replace $\mathbb{B}_R^n$ with the ball of radius $R$ in $\QD(X)$ (assuming $n=3g-3$ where $g$ is the genus of the Riemann surface $X$). In this case $\|v\|$ refers to the norm on $\QD(X)$, while the constant $C$  also depends on $X$ (although it is uniform in $X$ providing that $X$ belongs to some  compact subset of $\Teich$). 
\end{remark}

\subsection{Preparing the proof of Proposition \ref{proposition: bound on fiber} } 
We start by defining an initial neighbourhood of $W$ which will be refined to extract the resulting neighbourhood $V\subset \QDD$ of $W$ from the statement of Proposition  \ref{proposition: bound on fiber}.

\begin{definition} Let $V\subset \QDD$ be a dilation invariant set. 
We define $\Ne_r(V)\subset \QDD$ by letting $\psi\in \Ne_r(V)$ if  there exists $\phi\in V$  belonging to  the same $\QD(X)$ as $\psi$, and such that
$\|\psi-\phi\|<r\|\phi\|$. Note that $\Ne_r(V)$ is also dilation invariant.
\end{definition}

\begin{lemma}\label{lemma-okolina-1} There exists a dilation invariant neighbourhood $V\subset \QDD$ of a closed dilation invariant set $W\subset \B''$, and $r>0$, with the following properties
\begin{itemize}
\item[(1)] if $\QD(X)\cap \Ne_{r}(V)\ne \emptyset$ then the isomorphism $L:\QD(X)\to \B$ is $2$-bilipschitz,
\vskip .1cm
\item[(2)] $F$ is a local diffeomorphism on  $\dett^{-1}(\Ne_{r}(V))$,
\vskip .1cm
\item[(3)] the closure of the set $(\pi\circ F)\big(\dett^{-1}(\Ne_{r}(V))\big)$ is contained in $\B''$.
\end{itemize}
\end{lemma}
\begin{proof}  We begin by observing that linear isomorphisms  (trivialisations) $L:\QD(X)\to \B$ converge to the identity map $\B\to \B$ when $d_\Teich(X,X_0)\to 0$. Thus, we can find a dilation invariant neighbourhood $V_1\subset \QDD$ so that if $\QD(X)\cap V_1\ne \emptyset$ then
$L$ is $2$-bilipschitz.

On the other hand, by Lemma \ref{lemma-elem-1} there exists a dilation invariant neighbourhood $V_2\subset \QDD$ of $W$  so that $F$ is a local diffeomorphism on $\dett^{-1}(V_2)$. Moreover, using  Proposition \ref {proposition: WKB}  we can choose a dilation invariant neighbourhood $V_3\subset \QDD$ of $W$ so that the closure of the set $F(\dett^{-1}(V_3))$ is contained in $\B''$. Set $V_0=V_1\cap V_2 \cap V_3$. Then there exists $r>0$, and a dilation invariant neighbourhood $V$, so that  
$\Ne_{r}(V)\subset V_0$. This $V$ and $r$ satisfy the conditions (1), (2), and (3),  of the lemma.
\end{proof}

In the remainder of the section we fix $V_0$ and $r_0$ satisfying the conclusions of the previous lemma. 
Suppose $\psi_0\in \QD(X)\cap V_0$ be of norm one, and let $0<r\le r_0$. Consider the ball $B_{rt}(t\psi_0)$. Since $B_{rt}(t\psi_0)$ is contractible and contained in $\Ne_{r}(V_0)$, it follows that  $\dett^{-1}$ has finitely many inverse branches on $B_{rt}(t\psi_0)$ (see Lemma \ref{lemma-elem-1}). We choose one of them and denote it by $\dett^{-1}:B_{rt}(t\psi_0)\to \Syst$. We define the map 
$$
f:B_{rt}(t\psi_0)\to \M
$$ 
by letting $f=F\circ \dett^{-1}$.
Then inequality (\ref{eq: bound on fiber}) from Proposition \ref{proposition: bound on fiber}  is equivalent to
\begin{equation}\label{eq-bf2}
\|df(\dot{\psi})\|_{g_{L^{2}}}\le C\frac{\|\dot{\psi}\|}{\sqrt{t}},\quad\quad \forall \dot{\psi}\in \QD(X).
\end{equation}
Thus, it remains to find a dilation invariant neighbourhood $V\subset V_0$ of $W$, and $0<r\le r_0$, so that (\ref{eq-bf2}) holds  for every such map $f$, and every $\dot{\psi}$, assuming $t$ is large enough.

\subsection{Proof of Proposition \ref{proposition: bound on fiber} } We need  the following lemma
before we can employ the Yau-Schwartz lemma to prove (\ref{eq-bf2}). Furthermore,  in this lemma we 
define the neighbourhood $V\subset \QDD$ of $W$ from the statement of Proposition \ref{proposition: bound on fiber}.

We let $Z$ denote the set of norm one differentials in the closure of the set 
$(\pi\circ F)\big(\dett^{-1}(\Ne_{r}(V))\big)$. Then $Z$ is a closed subset  of $\B''$. Let $s>0$  denote the  constant $r$ from  Proposition \ref{proposition: curvature} (for this choice of the set $Z$).

\begin{lemma}\label{lemma-dokaz} There exists a dilation invariant neighbourhood
$V\subset V_0 \subset \QDD$  of $W$, and  constants $r,t_0>0$, so that for every  norm one 
$\psi_0\in \QD(X)\cap V$ we have that 
$$
(\pi\circ f)\big(B_{rt}(t\psi_0)\big)\subset B_{st}((\pi\circ f)(t\phi_0)) 
$$
assuming $t\ge t_0$.
\end{lemma}

\begin{proof}  From Proposition \ref {proposition: WKB} we can find a dilation invariant neighbourhood $V_1\subset \QDD$ of $W$, and $t_1>0$, so that 
\begin{equation}\label{eq-zoost}
\left\|4(\pi\circ f)(\psi)- L\big(\psi\big) \right\|\le  \frac{s}{100} \|L\big(\psi\big)\|,
\end{equation}
for every $\psi\in V_1(t)$ assuming $t\ge t_1$.

Let $V_2=V_0\cap V_1$. Suppose $\psi_1,\psi_2\in \QD(X)\cap V_2$. Combining the triangle inequality with (\ref{eq-zoost}), and the fact that $L$ is 2-bilipschitz,  yields
\begin{align*}
&\quad\quad\quad\quad\quad\quad\quad\quad\quad \left\|(\pi\circ f)(\psi_1)-(\pi\circ f)(\psi_2) \right\|= \\ 
&\left\|\left((\pi\circ f)(\psi_1)-\frac{L(\psi_1)}{4}\right)-\left((\pi\circ f)(\psi_2)-\frac{L(\psi_2)}{4}\right)+
\frac{1}{4}L(\psi_1-\psi_2)\right\| \\
&\quad\quad\quad\quad\quad\quad\quad\quad\quad \le \frac{s}{200} \left(\|\psi_1\| + \|\psi_1\|\right)+\frac{1}{2}\|(\psi_1-\psi_2)\|.\\
\end{align*}

The previous estimate implies that 
$$
(\pi\circ f)\big(B_{rt}(t\psi_0)\big)\subset B_{st}((\pi\circ f)(t\phi_0)) 
$$
as long as $r<s$, and $B_{rt}(t\psi_0)\subset V_2$. It remains to observe that there exists a dilation invariant neighbourhood $V$ of $W$, and $r>0$, such that $\Ne_r(V)\subset V_2$.
\end{proof}

Suppose $t\ge t_0$. The map 
$$
f:B_{rt}(t\psi_0)\to \pi^{-1}\big( B_{st}((\pi\circ f)(t\phi_0))  \big)
$$ 
is holomorphic with respect to the restriction of the complex structure $J$ to $\pi^{-1}\big( B_{st}((\pi\circ f)(t\phi_0))  \big)$. We apply Lemma \ref{lemma-yau} to $f$ (see the remark below that lemma), and obtain the inequality
$$
\|df(\dot{\psi})\|_{g_{J}}\le C_1\frac{\|\dot{\psi}\|}{rt},
$$
where $C_1$ is the constant from Lemma \ref{lemma-yau}. On the other hand, from
Proposition \ref{proposition: curvature} we have the estimate
$$
\|df(\dot{\psi})\|_{g_{L^{2}}} \le C_2 \sqrt{t} \|df(\dot{\psi})\|_{g_{J}}
$$
for some constant $C_2$. Putting the last two estimates together proves (\ref{eq-bf2}) by letting $C=C_1C_2/r$.

\section{Constructing the negatively curved K\"ahler metric}
In this section we recall the notion of the semiflat metric on $\M'$ and its connection with the Hitchin $L^2$ metric. In particular, we prove Lemma \ref{lemma: diam bound} and  Proposition \ref{proposition: curvature}. Both of them are statements  about the Hitchin metric $g_{L^{2}}$. The idea is to prove each of them for the semiflat metric first, and then use the fact that $g_\sf$ and $g_{L^{2}}$ are asymptotic to each other to deduce the corresponding statement for $g_{L^{2}}$.

\subsection{The two hyperk\"ahler structures} Beside  $(g_{L^2},I,J,K)$, there is another hyperk\"ahler structure $(g_{\sf},I,J_{\sf},K_{\sf})$ on $\M'$. The two hyperk\"ahler structures share the same complex structure $I$ and the same holomorphic symplectic form 
$$
\eta_I = \omega_J + i \omega_K = \omega_{J_{\sf}} + i \omega_{K_{\sf}}.
$$ 
So it is convenient to adopt the equivalent definition of a hyperk\"ahler structure which consists of a complex symplectic manifold $(M,I,\eta)$ together with a K\"ahler metric $g$ compatible with $\eta$. Thus, the two hyperk\"ahler structures on $\M'$ have the same underlying complex symplectic manifold $(\M',I,\eta_I)$, but they have  two different K\"ahler metrics $g_{L^2}$ and $g_{\sf}$.

By Proposition 2.1 in \cite{m-s-w-w} there is a symplectomorphism between $(\M',\eta_I)$ and $(T^*\B'/\Gamma,\eta)$. Here $\Gamma=\bigcup_{\phi\in\B'}\Gamma_\phi$ is a local system over $\B'$, and $\eta$ is covered by the canonical symplectic form on the cotangent bundle $T^*\B'$.
Under the identification $\M'\cong T^*\B'/\Gamma$, the metric $g_{\sf}$ is covered by a hyperk\"ahler metric on $T^*\B'$ which is also called the semiflat metric and denoted $g_{\sf}$. 

The semiflat metric $g_{\sf}$ on $T^*\B'$ is naturally associated to the special K\"ahler metric $g_{\sk}$ on $\B'$ through the holomorphic Lagrangian fibration $\pi:\M'\to \B'$. We point to \cite{freed} as the standard reference for special K\"ahler manifolds and the associated hyperk\"ahler metric on their cotangent bundle (see also \cite{m-s-w-w} Section 2.4.1).

Mazzeo-Swoboda-Weiss-Witt \cite{m-s-w-w} proved that the two metrics are asymptotic to each other. The  relation between $g_{L^2}$ and $g_{\sf}$ is given by the following  result  (see  Theorem 1.2 in \cite{m-s-w-w}).

\begin{theorem}\label{theorem: asymptotic} Suppose $W\subset \B'$ is a closed, dilation invariant subset. Then the estimate 
$$
g_{L^2}=g_{\sf}\cdot(1+o(1))
$$
holds on $\pi^{-1}(W(t))\subset \M'$, where $o(1)\to 0$ with $t\to \infty$.
\end{theorem}
\begin{remark} In \cite{m-s-w-w} it was showed that the error term decays polynomially. Dumas-Neitzke \cite{d-n}, and Fredrickson \cite{fredrickson}, improved this estimate by showing that the decay is exponential. But for our purposes only the weakest form stated above is required.
\end{remark}

\subsection{Proof of Lemma \ref{lemma: diam bound}} 
For $\phi\in \B'<\QD(X_0)$, we let $Y_\phi$ denote the corresponding spectral curve which is a double cover of $X_0$ defined so that the lift of $\phi$ to $Y_\phi$ is a square of an Abelian differential. The fibre $M_{\phi}=\pi^{-1}(\phi)$ is isomorphic to  $\text{Prym}(Y_\phi)$. Moreover, the semiflat metric on $M_\phi \cong \text{Prym}(Y_\phi)$ agrees with the standard $L^2$ norm on the space of harmonic 1-forms on $Y_\phi$. 

Since the spectral curve $Y_{t\phi}$ does not depend on $t$, it follows that fibres $M_{t\phi}$ are mutually isometric with respect to $g_\sf$. So, there exists a constant $D_\phi>0$ so that the diameter of the fibre $M_{t\phi}$ is bounded above by $D_\phi$ for every $t$. Moreover, given  a closed dilation invariant sector $W\subset \B'$, we can find $D_0>0$ so that the $g_\sf$-diameters of the fibres $M_{t\phi}$, where $t>0$ and $\phi\in W$, are bounded above by $D_0$.

On the other hand, the metrics $g_\sf$ and $g_{L^{2}}$ are asymptotic along the rays $t\phi$, where $\phi\in W$, and 
$t \to +\infty$. Thus, there exists $D>0$ so that the $g_{L^{2}}$-diameters of the fibres $M_{t\phi}$, where $t\ge1$ and $\phi\in W$, are bounded above by $D$.

\subsection{The scaling action}
Let $t>0$. We define the action $\alpha_t:\B'\to \B'$ by $\alpha_t(\phi)=t\phi$, where $\phi\in \B'$. An important feature of the special K\"ahler metic $g_{\sk}$ on $\B'$ is that it is 'conic' (see for example page 159 in \cite{m-s-w-w}) which means that
\begin{equation}\label{eq-alpha}
\alpha_t^*(g_{\sk})=t\cdot g_{\sk}.
\end{equation}
We define the action $\beta_t:T^*\B'\to T^*\B'$  by 
$$
\beta_t(\phi,p)=(\alpha_t(\phi), \, t\cdot(\alpha_t^{-1})^*(p))
$$
where $p$ is a point in $T^*_\phi \B'$. The following lemma is proved in the appendix.

\begin{lemma}\label{lemma-beta} The map $\beta_t$ preserves the three complex structures $(I,J_\sf,K_\sf)$ on $T^*\B'$ associated to $g_{\sf}$. Moreover, we have
\begin{equation}\label{eq-beta}
\beta_t^*(g_{\sf})=t\cdot g_{\sf}.
\end{equation}
\end{lemma}

\begin{remark}
The action $\beta_t$ is not a lift  of the natural action on $\M'\cong T^*\B'/\Gamma$ given by $(E,\theta)\to (E,\sqrt{t}\theta)$. In fact, the map $\beta_t:T^*\B'\to T^*\B'$ does not descend to the quotient $T^*\B'/\Gamma \cong \M'$. As a matter of fact, the map $\beta_t$ is uniquely determined by the requirement that it is an extension of $\alpha_t$, and that it preserves the  complex structures $(I,J_\sf,K_\sf)$.
\end{remark}

We now pull back $(g_{L^{2}},I,J,K)$ by $\beta_t$ and get
$$
g^t=\beta_t^*g_{L^{2}},\quad I^t=\beta_t^* I,\quad J^t=\beta_t^* J ,\quad K^t=\beta_t^* K.
$$
As we already observed $I^t=I$. Furthermore, the following holds:
\begin{lemma}\label{lemma-conv} We have
$$
t^{-1}\cdot g^t \to g_{\sf} \quad \quad \text{and}\quad\quad J^t\to J_\sf
$$
in $C^\infty$ topology on compact subset of $T^*\B'$, when $t\to \infty$.
\end{lemma}
\begin{proof} It follows from 
Theorem \ref{theorem: asymptotic} and (\ref{eq-beta}) that  that $t^{-1}\cdot g^t \to g_{\sf}$ in $C^0$ sense on compact sets in  $T^*\B'$. 
In fact the convergence is in $C^\infty$ sense  due to the following lemma which is a consequence of the Calabi-Yau $C^3$ estimate. It is implicit in \cite{yau-1}, and in the present form can be recovered (for example) from Proposition 3.1 in \cite{h-t}. 
\begin{lemma}   Let $(M,I)$ be a complex manifold and let $g^t$ be a family of K\"ahler metrics with vanishing Ricci curvature such that $g^t\rightarrow g^\infty$ uniformly for some K\"ahler metric $g^\infty$. Then $g^t\rightarrow g^\infty$ in $C^\infty$ topology.
\end{lemma}
Since $J^t$ is uniquely determined by $g^t$ we find that $J^t\to J_\sf$ as well.
\end{proof}

\subsection{Metric of negative holomorphic sectional curvature on a fixed neighbourhood}
The following  lemma is readily deduced from a similar statement regarding the standard hyperk\"ahler structure on $T^*\C^n$. It is proved the appendix.

\begin{lemma}\label{lemma: curvature}
For each $\phi_0\in \B'$, with $\|\phi_0\|=1$, there are constants $C,\kappa>0$,  a neighbourhood 
$U\subset T^*\B'$ of the point $(\phi_0,0)\in  T^*\B'$, and a positive function $\varphi$ on $\pi(U)$,  with the following properties: 
\begin{itemize}
\item[(a)] the  $(1,1)$-form $\sqrt{-1}\partial_{J_\sf}\overline\partial_{J_\sf}(\varphi\circ\pi)$ is a K\"ahler form, and thus defines a K\"ahler metric (denoted $g_{J_{\sf}}$) on $U$, where $\pi:T^*\B'\to \B'$ is the  projection,
\vskip .1cm
\item[(b)] the  holomorphic sectional curvature of $g_{J_{\sf}}$ is  $<-\kappa$, 
\vskip .1cm
\item[(c)] $C^{-1} g_\sf < g_{J_{\sf}} < Cg_\sf$.
\end{itemize}
\end{lemma}

\subsection{Proof of Proposition \ref{proposition: curvature}}

Fix $\phi_0\in \B'$, with $\|\phi_0\|=1$. Let $U$ be the neighbourhood, and $\varphi$   the  positive function on $\pi(U)$ from Lemma \ref{lemma: curvature}. Furthermore, let  $V\subset T^*\B'$ be a neighbourhood of the point 
$(\phi_0,0)\in  T^*\B'$ so that the closure of $V$ is contained in $U$.
Set $U_t=\beta_t(U)$, $V_t=\beta_t(V)$, and $\varphi_t=\varphi\circ \beta^{-1}_t$. Then $V_t\subset U_t \subset T^*\B'$  are neighbourhoods of the point $(t\phi_0,0)\in  T^*\B'$, and $\varphi_t$ a positive valued function on $\pi(U_t)$.

\begin{claim}\label{claim: curvature-1} Let $C,\kappa>0$ be the constants from Lemma \ref{lemma: curvature}. Then for every large enough $t$ we have
\begin{itemize}
\item[(a)] the  $(1,1)$-form $\sqrt{-1}\partial_{J}\overline\partial_{J}(\varphi_t\circ\pi)$ is a K\"ahler form, and thus defines a K\"ahler metric (denoted $g_{J}$) on $V_t$, where $\pi:T^*\B'\to \B'$ is the  projection,
\vskip .1cm
\item[(b)] the  holomorphic sectional curvature of $g_{J}$ is  $<-\kappa$ on $V_t$, 
\vskip .1cm
\item[(c)] $C^{-1} g_{L^{2}} < tg_{J} < Cg_{L^{2}}$ on $V_t$.
\end{itemize}
\end{claim}
\begin{proof}
Combining Lemma \ref{lemma-conv} with Lemma \ref{lemma: curvature} shows that for large enough $t$ the $(1,1)$-form $\sqrt{-1}\partial_{J^{t}}\overline\partial_{J^{t}}(\varphi \circ \pi)$ is a K\"ahler form, and it therefore  defines a K\"ahler metric  $g_{J^{t}}$ on $V$. Furthermore, the  holomorphic sectional curvature of $g_{J^{t}}$ is $<-\kappa$, and the inequalities $C g^t < g_{J^{t}} < C^{-1}g^t$ holds, for large enough $t$. Since $\beta_t^*J=J^t$, and $\beta_t^* g_J=g_{J^{t}}$, the claim follows.
\end{proof}

We are ready to finish the proof of Proposition \ref{proposition: curvature}. Firstly, there exists $r>0$ such that $B_{r}(\phi_0)=\{\phi\in \B:\, \|\phi-\phi_0\|<r\}$ is contained in $V\cap \B'$.
Thus,  the ball $B_{rt}(t\phi_0)$ is contained in $V_t\cap \B'$. Secondly, since the diameter of the fibre $\M_\phi=\pi^{-1}(\phi)$  is less than some constant $D$ for each $\phi \in B_{rt}(t\phi_0)$ (here $D$ does not depend on $t$ according to Lemma \ref{lemma: diam bound}), it follows that for every $q\in \M_\phi$ there exists a point $\wt{q}\in V_t$ such that $\pi(\wt{q})=q$ (recall that $T^*B_{rt}(t\phi_0)/\Gamma \cong \pi^{-1}\big(B_{rt}(t\phi_0)\big)$).

On the other hand, observe that the $(1,1)$-form $\sqrt{-1}\partial_{J}\overline\partial_{J}(\varphi_t\circ\pi)$ is well defined on $\pi^{-1}(B_{rt}(t\phi_0))$. Since $V_t$ contains a fundamental domain of the covering 
$T^*B_{rt}(t\phi_0)\to  \pi^{-1}\big(B_{rt}(t\phi_0)\big)$, we conclude that all three conditions  of the metric $g_J$ from Claim \ref{claim: curvature-1} which hold on $V_t$ also hold on  $\pi^{-1}\big(B_{rt}(t\phi_0)\big)$ when $t$ is large enough. This completes the proof of the proposition.

\appendix

\section{Proof of Lemma \ref{lemma-beta}}   
To prove the lemma, it suffices to prove the following more general claim.
\begin{claim}
    Let $\alpha: M\rightarrow M'$ be a diffeomorphism between special K\"ahler manifolds $(M,g,\nabla)$ and $(M',g',\nabla')$ such that
    \[\alpha^*(g')=t\cdot g,\quad\quad \alpha^*(\nabla')=\nabla.\]
    Define $\beta:T^*M\rightarrow T^*M'$ by
    \[\beta(z,p)=(\alpha(z),t\cdot(\alpha^{-1})^*(p)).\]
    Then $\beta$ respects the  three induced complex structures on $T^*M$ and $T^*M'$ respectively, and
    \[\beta^*(g_{\sf}')=t\cdot g_{\sf}.\]
\end{claim}
\begin{proof}
    This is a linear algebra statement. Let $a:V\rightarrow V'$ be an isomorphism between vector spaces  equipped with Hermitian metrics $g,g'$  such that $a^*(g')=t\cdot g$. Define $b:T^*V\rightarrow T^*V'$ by
    \[b(v,p)=(a(v),t\cdot(a^{-1})^*(p)).\]
    Then
    \[b^*(g_{\sf}')=t\cdot g_{\sf}.\]
    To see this, we identify $V$ and $V'$ via $a$, that is,  we assume $V'=V$, and $a=\id_V$. The two Hermitian metrics on $V$ satisfy $g'=t\cdot g$. Then the metrics $g^{-1}$ and $g'^{-1}$ on $V^*$ satisfy $g'^{-1}=t^{-1}\cdot g^{-1}$. 
    
On the other hand,  $g_{\sf}$ is defined by taking the product of $g$ and $g^{-1}$ on $T^*V\cong V\times V^*$, and likewise for $g'_{\sf}$. Since $b:T^*V\rightarrow T^*V$ is given by $b(v,p)=(v,t\cdot p)$, it follows that  
$b^*(g_{\sf}')=t\cdot g_{\sf}$ (note that on the level of 2-tensor $t\cdot p$ means scaling by $t^2$ in the fibre direction).
\end{proof}

\section{Proof of Lemma \ref{lemma: curvature}}

\subsection{Flat metric}
Consider $\C^{2n}$ with coordinates $(z_1,\cdots,z_n,w_1,\cdots,w_n)$, where $z_j,w_j\in \C$.
Equip $T^*\C^n \cong \C^{2n}$ with the standard hyperk\"ahler structure $(g_0,I_0,J_0,K_0)$. 
Define the cube $R^n\subset \C^n$ by  \[R^n=\left\{(z_1,\cdots,z_{n})\in\C^{n}\big|-\frac{\pi}{2}<\re(z_1),\mathrm{Im}(z_1),\cdots,\re(z_n),\mathrm{Im}(z_n)<\frac{\pi}{2}\right\}.\]

Define the infinite strips 
$$
S^1_k=\left\{-\frac{\pi}{2}<\re(z_k)<\frac{\pi}{2}\right\} \times \{\mathrm{Im}(w_k)=0\},
$$ 
and 
$$
S^2_k=\left\{-\frac{\pi}{2}<\mathrm{Im}(z_k)<\frac{\pi}{2}\right\} \times \{\re(w_k)=0\}.
$$
Then
\begin{equation}\label{eq-proda}
T^*R^n=\prod_{k=1}^{n} S^1_k \times S^2_k.
\end{equation}
Endow each strip $S^i_k$ with the complete hyperbolic metric, and let $g_{J_{0}}$ denote the induced product metric on $T^*R^n$  (with respect to product structure (\ref{eq-proda})).
Since each $S^i_k$ is $J_0$ holomorphic in $T^*R^n$, it follows that $g_{J_{0}}$ is a K\"ahler metric on the complex manifold $(T^*R^n,J_0)$. In particular, $(T^*R^n,g_{J_{0}})$ is isometric to $\Ha^{2n}$ which is the product of $2n$ hyperbolic planes $\Ha$.

Consider the following function $\varphi_0$ on $R^n$
\begin{align}\label{eq8}
    \varphi_0(z_1,\cdots,z_{n})=-\sum_{k=1}^n\left[\log\cos(\re(z_k))+\log\cos(\mathrm{Im}(z_k))\right].
\end{align}
Let  $\pi_0:T^*R^n \to R^n$ be the natural projection. By $\BB^n_s\subset \C^n$ we denote the  ball of radius $s$ centred at the origin $0$.

\begin{lemma}\label{lemma: curvature model} There exist constants $\kappa,C>0$ with the following properties. The $(1,1)$-form $\sqrt{-1}\partial_{J_0}\bar\partial_{J_0}(\varphi_0\circ\pi_0)$ is a K\"ahler form on $T^*R^n$. The corresponding K\"ahler metric is equal to $g_{J_0}$, and it has holomorphic sectional curvature $<-\kappa$. Moreover, the inequalities $C^{-1}g_0< g_{J_0}< C^{-1}g_0$ hold on $T^*\BB^n_1$.
\end{lemma}
\begin{proof} We leave to the reader to check that $\sqrt{-1}\partial_{J_0}\bar\partial_{J_0}(\varphi\circ\pi)$ is a K\"ahler form, and that the corresponding K\"ahler metric  equals $g_{J_0}$. Recall that the holomorphic sectional curvature on the product Poincar\`e metric on $\mathbb{H}^{2n}$ varies in the range $[-1,-1/n]$. Thus, it suffices to choose $0<\kappa<1/n$. The last assertion follows from the fact that $\BB^n_1$ is compactly contained in $R^n$. 
\end{proof}

\subsection{Semiflat metric}
Let $B$ be a special K\"ahler manifold and equip $T^*B$ with the associated hyperk\"ahler structure $(g_\sf,I,J,K)$. 
Fix $b \in B$. Then there exists a small enough $s>0$, and a local holomorphic chart $f:\BB^n_s \to B$, with $f(0)=b$,   such that  $df_0: \C^n \to T_{b} B$ is unitary. The following lemma follows  from Lemma \ref{lemma: curvature model} and the fact that the values of  $(g_\sf,I,J,K)$, and the values of $f_*(g_0,I_0,J_0,K_0)$, agree at the point  $(b,0)\in T^*B$.

\begin{lemma}\label{lemma: curvature semiflat}
There exists a neighbourhood  $U\subset T^*B$ of the point $(b,0)\in  T^*B$ such that
the $(1,1)$-form $\sqrt{-1}\partial_J\bar\partial_J(\varphi\circ f^{-1} \circ \pi)$ is a K\"ahler form on $U$. The induced K\"ahler metric $g_J$ has holomorphic sectional curvature $<-\kappa$. Moreover, $C^{-1}g_{\sf}< g_J< C^{-1}g_{\sf}$.
\end{lemma}
\begin{remark}
    In fact the same holds on $T^*(B\cap U)$, and not just on $U$, due to the fact that there is a group acting on $T^*(B\cap U)$ preserving the hyperk\"ahler structure and acting transitively on the fibres. But for our application we are content with the above lemma.
\end{remark}

Applying the previous lemma to the case $B=\B'$ yields the proof of Lemma \ref{lemma: curvature}

\end{document}